\documentclass[12pt]{article}

\usepackage[utf8]{inputenc}
\usepackage{amssymb}
\usepackage{amsmath}
\usepackage{mathtools}
\usepackage{amsfonts}
\usepackage{amsthm}
\usepackage{bbm}
\usepackage{bm}
\usepackage{hyperref}
\usepackage{physics}

\newcommand{\transpose}[1]{\ensuremath{\prescript{t}{}{#1}}}

\newtheorem{theorem}{Theorem}

\newtheorem{lemma}[theorem]{Lemma}
\newtheorem{proposition}[theorem]{Proposition}

\theoremstyle{definition}

\theoremstyle{remark}

\begin{document}

\title{Parameterizing roots of polynomial congruences}
\author{Matthew Welsh\thanks{Research supported by EPSRC grant EP/S024948/1}}
\date{6 July 2021}

\maketitle

\begin{abstract}
  We use the arithmetic of ideals in orders to parameterize the roots $\mu \pmod m$ of the polynomial congruence $F(\mu) \equiv 0 \pmod m$, $F(X) \in \mathbb{Z}[X]$ monic, irreducible and degree $d$.
  Our parameterization generalizes Gauss's classic parameterization of the roots of quadratic congruences using binary quadratic forms, which had previously only been extended to the cubic polynomial $F(X) = X^3 - 2$.
  We show that only a special class of ideals are needed to parameterize the roots $\mu \pmod m$, and that in the cubic setting, $d = 3$, general ideals correspond to pairs of roots $\mu_1 \pmod{m_1}$, $\mu_2 \pmod {m_2}$ satisfying $\gcd(m_1, m_2, \mu_1 - \mu_2) = 1$.
  At the end we illustrate our parameterization and this correspondence between roots and ideals with a few applications, including finding approximations to $\frac{\mu}{m} \in \mathbb{R}/ \mathbb{Z}$, finding an explicit Euler product for the co-type zeta function of $\mathbb{Z}[2^{\frac{1}{3}}]$, and computing the composition of cubic ideals in terms of the roots $\mu_1 \pmod {m_1}$ and $\mu_2 \pmod {m_2}$. 
\end{abstract}

\newpage

\tableofcontents

\newpage

\section{Introduction}
\label{sec:intro}

Let
\begin{equation}
  \label{eq:Fdef}
  F(X) = X^d + a_1 X^{d-1} + \cdots + a_d \in \mathbb{Z}[X]
\end{equation}
be an irreducible polynomial.
We call the residue classes $\mu \pmod {m}$ satisfying $F(\mu) \equiv 0 \pmod m$ the roots of the polynomial congruence $F(\mu) \equiv 0 \pmod m$.
In the quadratic setting, $d = 2$, there has been a lot of interest in studying the statistical properties of the sequence of $\frac{\mu}{m} \in \mathbb{R} / \mathbb{Z}$.
For example the work of Hooley \cite{Hooley1963}, Iwaniec \cite{Iwaniec1978}, Bykovskii \cite{Bykovskii1984}, Hejhal \cite{Hejhal1986}, Sarnak \cite{Sarnak1990}, Duke, Friedlander, Iwaniec \cite{DukeFriedlanderIwaniec1995}, \cite{DukeFriedlanderIwaniec2012}, and Toth \cite{Toth1997} all concern the equidistribution of this sequence and important subsequences.
Other statistics, such as upper bounds in short intervals, see Fouvry, Iwaniec \cite{FouvryIwaniec1997} and Friedlander, Iwaniec \cite{FriedlanderIwaniec1998} have also been of interest.

Much less is known in the cubic and higher degree setting, $d \geq 3$.
Hooley \cite{Hooley1964} has proven that the $\frac{\mu}{m}$ are still equidistributed modulo $1$, however his technique has more to do with the Chinese remainder theorem than roots of congruences, see the recent work of Kowalski, Soundararajan \cite{KowalskiSoundararajan2020}.
Consequently, Hooley's results are not nearly strong enough for applications like those in \cite{Iwaniec1978}, \cite{DukeFriedlanderIwaniec1995}, and \cite{Toth1997}.

At the heart of all of the cited work on the roots of quadratic congruences is the parameterization, essentially due to Gauss, of the modulus $m$ and roots $\mu \pmod m$ by means of binary quadratic forms or, what's more or less the same, ideals in quadratic orders.
Further, in the strongest results on the roots of quadratic congruences, i.e. \cite{Bykovskii1984}, \cite{Hejhal1986}, \cite{DukeFriedlanderIwaniec1995}, and \cite{Toth1997}, this parameterization provides an entrance for the spectral theory of $\mathrm{SL}(2)$, thus giving spectacular applications of this theory to arithmetic.
One might hope to find generalizations of this classic parameterization in the cubic and higher degree settings, and in this way obtain statistical results on the roots of higher degree polynomial congruences that go beyond those of \cite{Hooley1964}.
This was attempted, for example, in another work of Hooley \cite{Hooley1978}, but even with a parameterization of the roots $\mu^3 \equiv 2 \pmod m$, Hooley needed to appeal to out-of-reach conjectures to conclude nontrivial results.
There has since been some unconditional results in this direction in the works of Heath-Brown \cite{Heath-Brown2000}, where the largest prime factor of $n^3 -2$ is considered, and \cite{Heath-Brown2001}, where it is proved that $x^3 - 2y^3$ is infinitely often a prime number.
Both of these results in a sense use the parameterization of \cite{Hooley1978}.
The more recent works \cite{Dartyge2015} and \cite{delaBreteche2015} extend the method of \cite{Heath-Brown2000} to special quartic polynomials, their first step being to develop a parameterization of the roots of these degree four polynomials.
We remark that their method seems to require that the polynomial has Galois group $\mathbb{Z}/ 2 \mathbb{Z} \times \mathbb{Z} / 2\mathbb{Z}$ even though, as we shall see, one can parameterize the roots even when the polynomial is not Galois.
Understanding this requirement as well as extending the results of \cite{Heath-Brown2000} to general cubic polynomials are interesting questions that are unfortunately beyond the scope of this paper.

Nevertheless, the goal of the present work is to first generalize the parameterization of \cite{Hooley1978} to general polynomials, second to refine the parameterization in the cubic setting to one the author hopes is more amenable to analysis, and third to illustrate how one can use these parameterizations with some applications.
As one might expect, we find it far easier in the higher degree setting to work with ideals in degree $d$ orders rather than $d$-ary $d$-ic forms.

To parameterize the roots of $F(\mu) \equiv 0 \pmod m$, we naturally work with the order $\mathbb{Z}[\alpha]$, where $\alpha$ is a root of $F(\alpha) = 0$, considered abstractly, for the moment.
It turns out that one can parameterize all the $m$ and roots $\mu \pmod m$ using only a special class of ideals $ I \subset \mathbb{Z}[\alpha]$: those with $\mathbb{Z}[\alpha] / I$ additively cyclic.
The following proposition, proved in \ref{sec:generalinvariantfactors}, gives a useful characterization of these ideals in terms of a particular $\mathbb{Z}$-basis of $I$. 

\begin{proposition}
  \label{proposition:invariantfactors}
  Let $I$ be the sublattice of $\mathbb{Z}[\alpha]$ with basis $\{\beta_1, \dots, \beta_d\}$ given by
  \begin{equation}
    \label{eq:genhnfbasis1}
    \beta_i = \sum_{j=1}^d b_{ij} \alpha^{d-j+1},
  \end{equation}
  where the matrix $B = (b_{ij})_{1\leq i,j\leq d}$ is in upper-triangular Hermite normal form, meaning $b_{ij} = 0$ if $j < i$, $b_{jj} > 0$, and $0\leq b_{ij} < b_{jj}$ for all $i < j$.
  Then for $I$ to be an ideal of $\mathbb{Z}[\alpha]$, it is necessary that $b_{ii}$ divides $b_{ij}$ and $b_{jj}$ for all $i \leq j \leq d$.
  In particular, if $I$ is an ideal, then the $b_{ii}$ are the invariant factors of $\mathbb{Z}[\alpha]/I$. 
\end{proposition}

This proposition implies that those ideals $I \subset \mathbb{Z}[\alpha]$ having $\mathbb{Z}[\alpha] / I \cong \mathbb{Z} / m\mathbb{Z}$ as additive groups have a basis $\{\beta_1, \dots, \beta_d \}$ of the form
\begin{equation}
  \label{eq:betabasis}
  \begin{pmatrix}
    \beta_1 \\
    \vdots \\
    \beta_{d-1} \\
    \beta_d
  \end{pmatrix}
  =
  \begin{pmatrix}
    1 & \cdots & 0 & * \\
    \vdots & \ddots & \vdots & \vdots \\
    0 & \cdots & 1 & * \\
    0 & \cdots & 0 & m
  \end{pmatrix}
  \begin{pmatrix}
    \alpha^{d-1} \\
    \vdots \\
    \alpha \\
    1
  \end{pmatrix}
  .
\end{equation}
Our first theorem establishes a correspondence between these special ideals and the roots $\mu \pmod m$.

\begin{theorem}
  \label{theorem:correspondence}
  Let $I \subset \mathbb{Z}[\alpha]$ be an ideal such that the quotient $\mathbb{Z}[\alpha]/I$ is additively cyclic.
  Then $I$ has a unique basis $\{\beta_1, \dots, \beta_d\}$ of the form
  \begin{equation}
    \label{eq:correspondence}
    \begin{pmatrix}
      \beta_1  \\
      \vdots \\
      \beta_{d-1} \\
      \beta_d 
    \end{pmatrix} =
    \begin{pmatrix}
      1 & \cdots & 0 & -\mu^{d-1} \\
      \vdots &\ddots & \vdots & \vdots \\
      0 & \cdots & 1 & -\mu \\
      0 & \cdots & 0 & m
    \end{pmatrix}
    \begin{pmatrix}
      \alpha^{d-1} \\
      \vdots \\
      \alpha \\
      1
    \end{pmatrix}
    .
  \end{equation}
  where $m > 0$ and $\mu$ is, in order for uniqueness to hold, considered as a residue class modulo $m$ satisfying the polynomial congruence $F(\mu) \equiv 0 \pmod m$.

  Conversely, given an integer $m> 0$ and $\mu \pmod m$ satisfying $F(\mu) \equiv 0 \pmod m$, the sublattice $I$ of $\mathbb{Z}[\alpha]$ given by the basis $\{\beta_1, \dots, \beta_d\}$ as in (\ref{eq:correspondence}) is an ideal such that $\mathbb{Z}[\alpha]/I$ is cyclic. 
\end{theorem}

We remark that in the quadratic case, requiring $\mathbb{Z}[\alpha] / I$ to be cyclic is equivalent to the ideal $I$ not having any rational integer divisors.
This restriction is thus relatively minor in this setting; any ideal can be factored uniquely as a rational integer times an ideal $I$ such that $\mathbb{Z}[\alpha] / I$ is cyclic.
Theorem \ref{theorem:correspondence} then gives a characterization of all ideals in the quadratic order $\mathbb{Z}[\alpha]$.
For example, it easily implies that the Dedekind zeta function,
\begin{equation}
  \label{eq:dedekindzeta}
  \zeta_{\mathbb{Z}[\alpha]}(s) = \sum_{I} N(I)^{-s}, \quad \mathrm{Re}(s) > 1,
\end{equation}
where the sum is over non-zero ideals $I$ of $\mathbb{Z}[\alpha]$ and $N(I)$ denotes the norm of $I$, can be expressed as
\begin{equation}
  \label{eq:dedekindzeta1}
  \zeta_{\mathbb{Z}[\alpha]}(s) = \zeta(2s)\sum_{m\geq 1} \frac{1}{m^s} \# \{\mu \pmod m : F(\mu) \equiv 0 \pmod m \}.
\end{equation}

When $d\geq 3$ however, $I$ not having rational integer divisors is necessary but not sufficient for $\mathbb{Z}[\alpha] / I$ to be cyclic.
This can be seen, for example, by considering a degree two prime or a product of distinct, conjugate degree one primes.
It is therefore of interest to find extensions of theorem \ref{theorem:correspondence} that give correspondences relating to more general classes of ideals.
We do not carry this out in generality here, but we do obtain a satisfactory result in the cubic setting, $d = 3$. 

\begin{theorem}
  \label{theorem:cubiccorrespondence}
  Let $I \subset \mathbb{Z}[\alpha]$ be an ideal that is not divisible by any rational integers. Then $I$ has a basis $\{\beta_1, \beta_2, \beta_3\}$ in the form
  \begin{equation}
    \label{eq:idealcorrespondence}
    \begin{pmatrix}
      \beta_1 \\
      \beta_2 \\
      \beta_3
    \end{pmatrix} =
    \begin{pmatrix}
      1 & \mu_1 + a_1 & \lambda \\
      0 & m_1 & -\mu_2 m_1 \\
      0 & 0 & m_1m_2
    \end{pmatrix}
    \begin{pmatrix}
      \alpha^2 \\
      \alpha \\
      1
    \end{pmatrix}
  \end{equation}
  where $a_1$ comes from (\ref{eq:Fdef}), $m_1$, $m_2$ are positive integers and $\mu_1 \pmod {m_1}$, $\mu_2 \pmod {m_2}$ satisfy
  \begin{equation}
    \label{eq:cubiccongruence2}
    F(\mu_1) \equiv 0 \pmod {m_1}, \quad F(\mu_2) \equiv 0 \pmod {m_2}.
  \end{equation}
  Moreover, if $\gcd(m_1,m_2,D) = 1$, where $D$ is the discriminant of $F$, then $\gcd(m,n,\mu_1 - \mu_2) = 1$ and the basis (\ref{eq:idealcorrespondence}) is unique modulo integral, upper-triangular unipotent matrices acting on the left. 
  
  Conversely, if $m_1$, $m_2$ are positive integers and $\mu_1 \pmod {m_1}$, $\mu_2\pmod {m_2}$ satisfy (\ref{eq:cubiccongruence2}) and $\gcd(m_1,m_2,\mu_1 - \mu_2) = 1$, then there exists a $\lambda \pmod {m_1m_2}$, unique for fixed representatives of $\mu_1\pmod{m_1}, \mu_2\pmod{m_2}$, such that the sublattice $I$ with basis $\{\beta_1, \beta_2, \beta_3\}$ given in (\ref{eq:idealcorrespondence}) is an ideal of $\mathbb{Z}[\alpha]$. 
\end{theorem}

Before stating our next theorem, which concerns the parameterization of the roots $\mu \pmod m$, we fix some notation.
First, we no longer think of $\alpha$ as an abstract solution to $F(\alpha) = 0$, but rather as a vector in $\mathbb{C}^d$ with coordinates the $d$ embeddings into $\mathbb{C}$ of the abstract root $\alpha$.
Of course this $\alpha$, and all of the rational expressions in $\alpha$, i.e. $\mathbb{Q}(\alpha)$, are contained in a smaller space, $C_{r_1, r_2} \subset \mathbb{C}^d$, which we call the signature space of $\alpha$.
Here $r_1$ is the number of real embeddings of $\alpha$ and $r_2$ is the number of conjugate pairs of complex embeddings, so $d = r_1 + 2r_2$.
$C_{r_1, r_2}$ is defined to be the subset of $\mathbb{C}^d$ where the first $r_1$ coordinates are real numbers, and the next $2r_2$ are arranged in conjugate pairs.
So if $\xi \in C_{r_1,r_2}$ and we denote the coordinates of $\xi$ by $\xi^{(1)}, \dots, \xi^{(d)}$, then we have $\xi^{(j)} \in \mathbb{R}$ for $1\leq j \leq r_1$, and $\xi^{(r_1 + j + 1)} = \overline{\xi^{(r_1 + j)}}$ for $1 \leq j \leq 2r_2$ odd.
We further denote by $C_{r_1,r_2}^+ \subset C_{r_1,r_2}$ the set of $\xi \in C_{r_1,r_2}$ for which the real coordinates $\xi^{(1)}, \dots, \xi^{(r_1)}$ are positive.

We let $h^+(\alpha)$ denote the narrow class number of $\mathbb{Z}[\alpha]$, and we fix integral ideal representatives $I_l$, $1\leq l\leq h^+(\alpha)$, of the narrow ideal classes, so every invertible ideal $I \subset \mathbb{Z}[\alpha]$ can be written as $I = \xi I_l$ for some $\xi \in I_l^{-1} \cap C^+_{r_1,r_2}$ and unique $l$.
We fix $\mathbb{Z}$-bases $\{\beta_{l1}, \dots, \beta_{ld}\}$ of the $I_l$ and we set
\begin{equation}
  \label{eq:Bldef}
  \mathfrak{B}_l =
  \begin{pmatrix}
    \beta_{l1} \\
    \vdots \\
    \beta_{ld}
  \end{pmatrix}
  =
  \begin{pmatrix}
    \beta_{l1}^{(1)} & \cdots & \beta_{l1}^{(d)} \\
    \vdots & \ddots & \vdots \\
    \beta_{ld}^{(1)} & \cdots  & \beta_{ld}^{(d)}
  \end{pmatrix}
  .
\end{equation}
We further stipulate that $I_1 = \mathbb{Z}[\alpha]$, and $\{\beta_{11}, \dots, \beta_{1d}\} = \{ \alpha^{d-1}, \dots, \alpha, 1\}$, so
\begin{equation}
  \label{eq:B1def}
  \mathfrak{B}_1 =
  \begin{pmatrix}
    (\alpha^{(1)})^{d-1} & \cdots & (\alpha^{(d)})^{d-1} \\
    \vdots & \ddots & \vdots \\
    \alpha^{(1)} & \cdots & \alpha^{(d)} \\
    1 & \cdots & 1
  \end{pmatrix}
  .
\end{equation}
We also assume that $\mathrm{sign} \det \mathfrak{B}_l = \mathrm{sign} \det \mathfrak{B}_1$ for all $l$.

Further, we set $\Gamma = \mathrm{SL}(d,\mathbb{Z})$ and
\begin{equation}
  \label{eq:Udef}
  U = \left\{
    \begin{pmatrix}
      1 & \cdots & 0 & * \\
      \vdots & \ddots & \vdots & \vdots \\
      0 & \cdots & 1 & * \\
      0 & \cdots & 0 & 1
    \end{pmatrix}
    \in \Gamma \right\}.
\end{equation}
Finally, for each $l$, $1\leq l \leq h^+(\alpha)$, we set
\begin{equation}
  \label{eq:Gammaldef}
  \Gamma_l = \Gamma \cap \left \{\mathfrak{B}_l 
    \begin{pmatrix}
      \xi^{(1)} & \cdots & 0 \\
      \vdots & \ddots & \vdots \\
      0 & \cdots & \xi^{(d)}
    \end{pmatrix}
    \mathfrak{B}_l^{-1} : \xi \in C_{r_1,r_2}^+ \right\}.
\end{equation}
We are now ready to state our theorem on the parameterization of the roots $\mu \pmod m$ of general polynomial congruences.

\begin{theorem}
  \label{theorem:parameterization}
  Let $m$ be a positive integer and $\mu\pmod m$ satisfy $F(\mu) \equiv 0 \pmod m$.
  We assume that this $m$ and $\mu\pmod m$ correspond via theorem \ref{theorem:correspondence} to an invertible ideal in $\mathbb{Z}[\alpha]$.
  Then there is a unique $l$, $1\leq l \leq h^+(\alpha)$, and unique double coset $U \gamma \Gamma_l \in U \backslash \Gamma / \Gamma_l$ such that
  \begin{equation}
    \label{eq:parameterization}
    \begin{pmatrix}
      1 & \cdots & 0 & -\mu^{d-1} \\
      \vdots & \ddots & \vdots & \vdots \\
      0 & \cdots & 1 & -\mu \\
      0 & \cdots & 0 & m
    \end{pmatrix}
    \mathfrak{B}_1
    = \gamma \mathfrak{B}_l
    \begin{pmatrix}
      \xi^{(1)} & \cdots & 0 \\
      \vdots & \ddots & \vdots \\
      0 & \cdots & \xi^{(d)}
    \end{pmatrix}
  \end{equation}
  for some $\xi \in C^+_{r_1, r_2}$.

  Conversely, given an $l$, $1\leq l \leq h^+(\alpha)$ and double coset $U\gamma \Gamma_l \in U \backslash \Gamma / \Gamma_l$ for which there exists $\xi \in C^+_{r_1, r_2}$ such that 
  \begin{equation}
    \label{eq:converseparameterization}
    \gamma \mathfrak{B}_l
    \begin{pmatrix}
      \xi^{(1)} & \cdots & 0 \\
      \vdots & \ddots & \vdots \\
      0 & \cdots & \xi^{(d)}
    \end{pmatrix}
    =
    \begin{pmatrix}
      1 & \cdots & 0 & * \\
      \vdots & \ddots & \vdots & * \\
      0 & \cdots & 1 & * \\
      0 & \cdots & 0 & * 
    \end{pmatrix}
    \mathfrak{B}_1 ,
  \end{equation}
  then necessarily
  \begin{equation}
    \label{eq:converseparameterization1}
    \gamma \mathfrak{B}_l
    \begin{pmatrix}
      \xi^{(1)} & \cdots & 0 \\
      \vdots & \ddots & \vdots \\
      0 & \cdots & \xi^{(d)}
    \end{pmatrix}
    =
    \begin{pmatrix}
      1 & \cdots & 0 & -\mu^{d-1} \\
      \vdots & \ddots & \vdots & \vdots \\
      0 & \cdots & 1 & -\mu \\
      0 & \cdots & 0 & m
    \end{pmatrix}
    \mathfrak{B}_1
  \end{equation}
  where $m$ is a positive integer and $\mu\pmod m$ satisfies $F(\mu) \equiv 0 \pmod m$.
  Moreover, such $m$ and $\mu \pmod m$, if they exist, are unique and correspond to an invertible ideal via theorem \ref{theorem:correspondence}.
\end{theorem}

There are many minor variations of theorem \ref{theorem:parameterization}.
For example, if one wanted to avoid the use of complex numbers, one could consider the real and imaginary parts of a complex embedding instead of the conjugate pair.
This would have the effect of replacing the diagonal matrix on the right of (\ref{eq:parameterization}) with $r_2$ $2\times 2$ blocks of scaling/rotating matrices along the bottom right of the diagonal.
A more significant variation on theorem \ref{theorem:parameterization} would be to write the modulus $m$ and root $\mu\pmod m$ explicitly in terms of the entries of $\gamma$, recovering both Gauss's classic parameterization of the roots of quadratic congruences and the parameterization in \cite{Hooley1978} for the specific example $F(X) = X^3 - 2$.
This is done in theorem \ref{theorem:parameterization1} below, and to state it we establish a little more notation.

For each $l$, $1\leq l \leq h^+(\alpha)$, we fix a basis $\{\overline{\beta}_{l1}, \dots, \overline{\beta}_{ld}\}$ of the ideal $I_l^{-1}$ and define the integers $b_{ijkl}$ by
\begin{equation}
  \label{eq:bijkldef}
  \overline{\beta}_{li}\beta_{lj} = \sum_{1\leq k \leq d} b_{ijkl} \alpha^{d-k}.
\end{equation}
We let $B_{li}$ be the matrix with entries $b_{ijkl}$, $1\leq j \leq d$ indexing the rows and $1\leq k\leq d$ indexing the columns.
We also fix a fundamental domain $\mathcal{D} \subset \mathbb{Q}(\alpha)\cap C_{r_1,r_2}^+$ for the action of the totally positive units in $\mathbb{Z}[\alpha]$. 

\begin{theorem}
  \label{theorem:parameterization1}
  For a given $l$, $1\leq l\leq h^+(\alpha)$, and integers $c_1, \dots, c_d$, set
  \begin{equation}
    \label{eq:Cdef}
    C = \sum_{1\leq i\leq d} c_i B_{li}, \quad \xi = \sum_{1\leq i\leq d} c_i \overline{\beta}_{li}.
  \end{equation}
  We assume that $\xi \in \mathcal{D}$ and that
  \begin{equation}
    \label{eq:gcdcondition}
    \gcd(\det C_{1d}, \det C_{2d}, \dots, \det C_{dd}) = 1,
  \end{equation}
  where $C_{ij}$ is the $(d-1)\times (d-1)$ minor of $C$ obtained by removing the $i$th row and $j$th column, so there exist integers $u_j$, $1\leq j \leq d$, such that
  \begin{equation}
    \label{eq:ujdef1}
    \sum_{1\leq j \leq d} (-1)^{j+d} u_j \det C_{jd} = 1.
  \end{equation}
  Then
  \begin{equation}
    \label{eq:gammadef}
    \gamma^{-1} =
    \begin{pmatrix}
      c_{11} & \cdots & c_{1(d-1)} & u_1 \\
      \vdots & \ddots & \vdots & \vdots \\
      c_{d1} & \cdots & c_{d(d-1)} & u_d
    \end{pmatrix}
    \in \Gamma,
  \end{equation}
   where $c_{ij}$ are the entries of $C$, satisfies
  \begin{equation}
    \label{eq:parameterization1}
    \gamma \mathfrak{B}_l
    \begin{pmatrix}
      \xi^{(1)} & \cdots & 0 \\
      \vdots & \ddots & \vdots \\
      0 & \cdots & \xi^{(d)}
    \end{pmatrix}
    =
    \begin{pmatrix}
      1 & \cdots & 0 & -\mu^{d-1} \\
      \vdots & \ddots & \vdots & \vdots \\
      0 & \cdots & 1 & -\mu \\
      0 & \cdots & 0 & m
    \end{pmatrix}
    \mathfrak{B}_1
  \end{equation}
  and
  \begin{equation}
    \label{eq:parameterization2}
    \begin{pmatrix}
      1 & \cdots & 0 & -\mu^{d-1} \\
      \vdots & \ddots & \vdots & \vdots \\
      0 & \cdots & 1 & -\mu \\
      0 & \cdots & 0 & m
    \end{pmatrix}
    = \gamma C,
  \end{equation}
  where $m$ is a positive integer and $\mu \pmod m$ satisfies $F(\mu) \equiv 0 \pmod m$.

  Moreover, given such $m$ and $\mu\pmod m$ satisfying $F(\mu)\equiv 0\pmod m$ and corresponding via theorem \ref{theorem:correspondence} to an invertible ideal, there exist unique integers $c_1, \dots, c_d$ corresponding to $m$ and $\mu \pmod m$ in the above way, with different choices of the integers $u_j$ corresponding to different representatives of the residue classes $\mu^j \pmod m$. 
\end{theorem}

We remark that if one wanted to prove theorem \ref{theorem:parameterization1} directly, skipping theorem \ref{theorem:parameterization}, there are many difficulties in the proof of theorem \ref{theorem:parameterization} that one can avoid, e.g. lemma \ref{lemma:B1}.
However, theorem \ref{theorem:parameterization} provides a geometric perspective that is useful in certain contexts.

Restricting our attention to the cubic setting $d=3$, in view of the extension theorem \ref{theorem:cubiccorrespondence} of theorem \ref{theorem:correspondence}, we can ask if there are similar extensions of theorems \ref{theorem:parameterization} and \ref{theorem:parameterization1}.
This is indeed the case, as shown in theorems \ref{theorem:cubicparameterization} and \ref{theorem:cubicparameterization1} below.
In theorem \ref{theorem:cubicparameterization} we let $\Gamma_\infty \subset \Gamma = \mathrm{SL}(3,\mathbb{Z})$ be the subgroup of unipotent (positive ones on the diagonal) upper-triangular matrices.

\begin{theorem}
  \label{theorem:cubicparameterization}
  Let $m_1$, $m_2$ be positive integers and $\mu_1\pmod {m_1}$, $\mu_2 \pmod {m_2}$ satisfy
  \begin{flalign}
    \label{eq:mu1mu2conditions}
    & F(\mu_1) \equiv 0 \pmod {m_1} \\ \nonumber
    & F(\mu_2) \equiv 0 \pmod {m_2} \\ \nonumber
    & \gcd(m_1, m_2, \mu_1 - \mu_2) = 1.
  \end{flalign}
  In addition, we assume that $\gcd(m_1m_2, D) = 1$. Then there is a unique $l$ and double coset $\Gamma_\infty \gamma \Gamma_l \in \Gamma_\infty \backslash \Gamma  / \Gamma_l$ such that
  \begin{equation}
    \label{eq:cubicparameterization}
    \begin{pmatrix}
      1 & \mu_1 + a_1 & * \\
      0 & m_1 & - \mu_2m_1 \\
      0 & 0 & m_1m_2
    \end{pmatrix}
    \mathfrak{B}_1 = \gamma \mathfrak{B}_l
    \begin{pmatrix}
      \xi^{(1)} & 0 & 0 \\
      0 & \xi^{(2)} & 0 \\
      0 & 0 & \xi^{(3)}
    \end{pmatrix}
  \end{equation}
  for some $\xi \in C_{r_1,r_2}^+$.

  Conversely, if $\Gamma_\infty \gamma \Gamma_l$ is such that
  \begin{equation}
    \label{eq:conversecubicparameterization}
    \gamma \mathfrak{B}_l
    \begin{pmatrix}
      \xi^{(1)} & 0 & 0 \\
      0 & \xi^{(2)} & 0 \\
      0 & 0 & \xi^{(3)}
    \end{pmatrix}
    =
    \begin{pmatrix}
      1 & * & * \\
      0 & * & * \\
      0 & 0 & *
    \end{pmatrix}
    \mathfrak{B}_1,
  \end{equation}
  then necessarily
  \begin{equation}
    \label{eq:conversecubicparameterization1}
        \gamma \mathfrak{B}_l
    \begin{pmatrix}
      \xi^{(1)} & 0 & 0 \\
      0 & \xi^{(2)} & 0 \\
      0 & 0 & \xi^{(3)}
    \end{pmatrix}
    =
    \begin{pmatrix}
      1 & \mu_1 + a_1 & * \\
      0 & m_1 & - \mu_2m_1 \\
      0 & 0 & m_1m_2
    \end{pmatrix}
    \mathfrak{B}_1,
  \end{equation}
  where $m_1$, $m_2$ are positive integers and $\mu_1 \pmod {m_1}$, $\mu_2 \pmod {m_2}$ satisfy (\ref{eq:mu1mu2conditions}).
  Moreover, these $m_1$, $m_2$ and $\mu_1 \pmod {m_1}$, $\mu_2 \pmod {m_2}$, if they exist, are unique. 
\end{theorem}

We can also give a variation on theorem \ref{theorem:cubicparameterization} along the lines of theorem \ref{theorem:parameterization1}.
To state this variation, theorem \ref{theorem:cubicparameterization1}, we make use of the Pl\"ucker coordinates of cosets $\Gamma_\infty \gamma \in \Gamma_\infty \backslash \Gamma$, which we describe as follows.
Given a representative $\gamma$ of $\Gamma_\infty \gamma$, if
\begin{equation}
  \label{eq:Pluckerdef}
  \gamma =
  \begin{pmatrix}
    * & * & * \\
    * & * & * \\
    a' & b' & c' 
  \end{pmatrix}
  , \quad \gamma^{-1} = 
  \begin{pmatrix}
    a & * & * \\
    b & * & * \\
    c & * & * 
  \end{pmatrix}
  ,
\end{equation}
then the Pl\"ucker coordinates of $\Gamma_\infty \gamma$ are $a,b,c,a',b',c'$.
These integers are well-defined and satisfy
\begin{flalign}
  \label{eq:Pluckerconditions}
  & \gcd(a,b,c) = 1 \\ \nonumber
  & \gcd(a',b',c') = 1 \\ \nonumber
  & aa' + bb' + cc' = 0.
\end{flalign}
Conversely, given integers satisfying (\ref{eq:Pluckerconditions}), there exists a unique coset $\Gamma_\infty \gamma$ such that any representative $\gamma$ satisfies (\ref{eq:Pluckerdef}).
For a reference see \cite{BumpFriedbergGoldfeld1988} for example.

We also introduce a little more  notation in addition to that used in stating theorem \ref{theorem:parameterization1}.
Since $I_l$ is an ideal, we can define integers $b_{ijkl}'$ by
\begin{equation}
  \label{eq:bijklprimedef}
  \beta_{li}\alpha^{3-j}  = \sum_{1\leq k \leq 3} b_{ijkl}' \beta_{lk},
\end{equation}
and we set $B_{li}'$ to be the matrix with entries $b_{ijkl}'$, $j$ indexing the rows and $k$ indexing the columns. 

\begin{theorem}
  \label{theorem:cubicparameterization1}
  For given $l$, $1\leq l \leq h^+(\alpha)$, and integers $c_1, c_2, c_3$, $c_1', c_2', c_3'$, set
  \begin{equation}
    \label{eq:Cxidef}
    C = \sum_{1\leq l\leq 3} c_i B_{li}, \quad \xi = \sum_{1\leq i \leq 3} c_i \overline{\beta}_{li}, \quad C' = \sum_{1\leq i \leq 3} c_i' B_i', \quad \xi' = \sum_{1\leq i\leq 3} c_i' \beta_{li}.
  \end{equation}
  We assume that
  \begin{flalign}
    \label{eq:ccprimeconditions}
    & \gcd(c_{11}, c_{21}, c_{31}) = 1 \\ \nonumber
    & \gcd(c_{31}', c_{32}', c_{33}') = 1 \\ \nonumber
    & \xi \xi' \in \mathbb{Z},
  \end{flalign}
  where $c_{ij}$, $c_{ij}'$ are respectively the entries of $C$, $C'$.
  We also assume that $\xi' \in \mathcal{D}$, a fixed fundamental domain for the action of the totally positive units on $C_{r_1, r_2}^+ \cap \mathbb{Q}(\alpha)$.
  Then the coset $\Gamma_\infty \gamma \in \Gamma_\infty \backslash \Gamma$ having Pl\"ucker coordinates $c_{11}, c_{21}, c_{31}, c_{31}', c_{32}', c_{33}'$ satisfies
  \begin{equation}
    \label{eq:cubicparameterization1}
    \gamma \mathfrak{B}_l
    \begin{pmatrix}
      \xi^{(1)} & 0 & 0 \\
      0 & \xi^{(2)} & 0 \\
      0 & 0 & \xi^{(3)}
    \end{pmatrix}
    =
    \begin{pmatrix}
      1 & \mu_1 + a_1 & * \\
      0 & m_1 & -\mu_2m_1 \\
      0 & 0 & m_1m_2
    \end{pmatrix}
    \mathfrak{B}_1
  \end{equation}
  and
  \begin{equation}
    \label{eq:cubicparameterization2}
    \begin{pmatrix}
      1 & \mu_1 + a_1 & * \\
      0 & m_1 & -\mu_2m_1 \\
      0 & 0 & m_1m_2
    \end{pmatrix}
    = \gamma C,
  \end{equation}
  where $m_1$, $m_2$ are positive integers and $\mu_1\pmod{m_1}$, $\mu_2 \pmod {m_2}$ satisfy
  \begin{equation}
    \label{eq:mu1mu2conditions2}
    F(\mu_1) \equiv 0 \pmod {m_1}, \quad F(\mu_2) \equiv 0 \pmod {m_2}.
  \end{equation}

  Conversely, given positive integers $m_1$, $m_2$ such that $\gcd(m_1m_2, D) = 1$ and residue classes $\mu_1 \pmod {m_1}$, $\mu_2\pmod {m_2}$ satisfying
  \begin{flalign}
    \label{eq:mu1mu1conditions3}
    & F(\mu_1) \equiv 0 \pmod {m_1} \\ \nonumber
    & F(\mu_2) \equiv 0 \pmod {m_2} \\ \nonumber
    & \gcd(m_1, m_2, \mu_1 - \mu_2) = 1,
  \end{flalign}
  there exists unique $l$ and integers $c_1,c_2,c_3$, $c_1',c_2',c_3'$ corresponding the $m_1, m_2$ and $\mu_1\pmod{m_1}$, $\mu_2\pmod {m_2}$ in the above way.
\end{theorem}

We note that in order to define the coset $\Gamma_\infty \gamma$ by the Pl\"ucker coordinates $c_{11},c_{21},c_{31},c_{31}', c_{32}', c_{33}'$ as above, these coordinates need to satisfy (\ref{eq:Pluckerconditions}).
These conditions, (\ref{eq:Pluckerconditions}), are in fact implied by the requirements (\ref{eq:ccprimeconditions}), as demonstrated in section \ref{sec:cubicparameterization1} below.
We remark however that (\ref{eq:ccprimeconditions}) contains an extra quadratic constraint on the coordinates in addition to the third line of (\ref{eq:Pluckerconditions}).
This shows that even with the extended correspondence, theorem \ref{theorem:cubiccorrespondence} over theorem \ref{theorem:correspondence}, a relatively small subset of cosets $\Gamma_\infty \gamma$ actually correspond to roots of the congruence.
This is a significant difference between the cubic and quadratic setting and is one reason why despite hope, see \cite{Terras1988} and \cite{Buttcane2012}, strong statistical results on the roots $\mu \pmod m$ have not been obtained using the spectral theory of $\mathrm{SL}(3)$ automorphic forms.

We illustrate \ref{theorem:cubicparameterization1} and the above remarks for the polynomial $F(X) = X^3 - 2$, obtaining as a consequence Hooley's parameterization \cite{Hooley1978}.
The ring $\mathbb{Z}[2^{\frac{1}{3}}]$ has class number $1$ and we set
\begin{equation}
  \label{eq:examplebeta}
  \begin{pmatrix}
    \beta_{11} \\
    \beta_{12} \\
    \beta_{13} 
  \end{pmatrix}
  =
  \begin{pmatrix}
    \overline{\beta}_{11} \\
    \overline{\beta}_{12} \\
    \overline{\beta}_{13}
  \end{pmatrix}
  =
  \begin{pmatrix}
    2^{\frac{2}{3}} \\
    2^{\frac{1}{3}} \\
    1
  \end{pmatrix}
  .
\end{equation}
We compute from (\ref{eq:bijkldef}) and (\ref{eq:bijklprimedef}) that
\begin{equation}
  \label{eq:exampleB}
  B_1 = B_1' =
  \begin{pmatrix}
    0 & 2 & 0 \\
    0 & 0 & 2 \\
    1 & 0 & 0 
  \end{pmatrix}
  ,\ B_2 = B_2' =
  \begin{pmatrix}
    0 & 0 & 2 \\
    1 & 0 & 0 \\
    0 & 1 & 0 
  \end{pmatrix}
  ,\ B_3 = B_3' =
  \begin{pmatrix}
    1 & 0 & 0 \\
    0 & 1 & 0 \\
    0 & 0 & 1
  \end{pmatrix}
  .
\end{equation}
For integers $c_1, c_2, c_3, c_1', c_2', c_3'$ we have
\begin{equation}
  \label{eq:exampleCCprime}
  C =
  \begin{pmatrix}
    c_3 & 2c_1 & 2c_2 \\
    c_2 & c_3 & 2c_1 \\
    c_1 & c_2 & c_3
  \end{pmatrix}
  ,\ C' =
  \begin{pmatrix}
    c_3' & 2c_1' & 2c_2' \\
    c_2' & c_3' & 2c_1' \\
    c_1' & c_2' & c_3' 
  \end{pmatrix}
  ,
\end{equation}
and so the conditions (\ref{eq:ccprimeconditions}) become $\gcd(c_1, c_2, c_3) = \gcd( c_1', c_2', c_3') = 1$ and
\begin{equation}
  \label{eq:exampleConditions}
  c_3c_1' + c_2 c_2' + c_1c_3' = 0,\quad 2c_1c_1' + c_3 c_2' + c_2 c_3' = 0.
\end{equation}
We note that the first constraint in (\ref{eq:exampleConditions}) implies that there is a well-defined coset $\Gamma_\infty \gamma \in \Gamma_\infty \backslash \Gamma$ with Pl\"ucker coordinates $c_1, c_2, c_3, c_1', c_2', c_3'$, but the second condition restricts to a rather thin subset of $\Gamma_\infty \backslash \Gamma$.
In addition, we note that (\ref{eq:exampleConditions}) implies that the vector $
\begin{pmatrix}
  c_1' & c_2' & c_3'
\end{pmatrix}
$ is proportional to the vector
\begin{equation}
  \label{eq:exampleProportional}
  \begin{pmatrix}
    c_2^2 - c_1c_3 & 2c_1^2 - c_2 c_3 & c_3^2 - 2c_1c_2
  \end{pmatrix}
  .
\end{equation}
In the case that the vector (\ref{eq:exampleProportional}) has coprime entries, the vectors must be equal up to sign, and one obtains Hooley's parameterization \cite{Hooley1978} from (\ref{eq:cubicparameterization2}).
A nice observation in this regard is that we in fact have
\begin{equation}
  \label{eq:examplem1}
  m_1 = \gcd(c_2^2 - c_1c_3, 2c_1^2 - c_2 c_3, c_3^2 - 2c_1c_2).
\end{equation}

We now move on to some applications of our correspondence and parameterization results.
We start with an approximation to the point $\left( \frac{\mu^{d-1}}{m}, \dots, \frac{\mu}{m} \right) \in \mathbb{R}^{d-1}/ \mathbb{Z}^{d-1}$ by one of $d$ explicit points that has rational coordinates all having the same denominator.

\begin{proposition}
  \label{proposition:generalapproximation}
  With the notation as in theorem \ref{theorem:parameterization}, let $C_{id}$ denote the $(d-1) \times (d-1)$ sub-matrix obtained from $C$ by removing the $i$th row and $d$th column, and set $\bm{u}_i$ to be the vector $(u_1,\dots,u_d)$ with the $i$th entry removed. Then for some $k$, $1\leq k \leq d$,
  \begin{equation}
    \label{eq:approximation}
    \begin{pmatrix}
      \frac{\mu^{d-1}}{m} \\ \vdots \\ \frac{\mu}{m} 
    \end{pmatrix} =
    C_{kd}^{-1}\bm{u}_k + O\left( \frac{1}{m} \right) \pmod{\mathbb{Z}^{d-1}},
  \end{equation}
  with the implied constant depending only on the polynomial $F$.
\end{proposition}

We note that the size of the denominator of the approximating point $C_{id}^{-1} \bm{u}_i$ is about $m^{1 - 1/d}$.
Relative to the error $O(\frac{1}{m})$, this approximation is of the same strength as that produced by Dirichlet's theorem on simultaneous Diophantine approximation.
The point then of proposition \ref{proposition:generalapproximation} is that we have an actual construction of the approximation as opposed to mere existence.
The utility of this is illustrated in the proof of the following proposition.

\begin{proposition}
  \label{proposition:spacing}
  Let $M$ be a positive real number and let $B$ be a ball in $\mathbb{R}^{d-1}/\mathbb{Z}^{d-1}$ with radius $\frac{1}{M}$. Then the number of $\left( \frac{\mu^{d-1}}{m}, \dots, \frac{\mu}{m} \right) \in B$ with $F(\mu)\equiv 0 \pmod m$ and $M<m\leq 2M$ is bounded by a constant depending only on the polynomial $F$. 
\end{proposition}

The proof of proposition \ref{proposition:spacing} relies on two main ingredients: that different approximating points in (\ref{eq:approximation}) are well-spaced from each other and that not too many different points $(\frac{\mu^{d-j}}{m})$ correspond to the same approximating point.
Proving the first claim uses a general fact that rational points in $\mathbb{R}^{d-1}$ with the same denominator are well-spaced unless they are contained in rational hyperplanes of small height, see \cite{CheungChevallier2016} and lemma \ref{lemma:torsionpointspacing}, and so it suffices to show that the points (\ref{eq:approximation}) are not contained in such rational hyperplanes.
On the other hand, proving the second claim uses the rational hyperplanes that do contain the approximating point to show that the map taking the point $(\frac{\mu^{d-j}}{m})$ to the approximation is $O(1)$ to $1$.
Obviously neither claim could be verified with only the existence of the approximation given by Dirichlet's theorem.

In a different direction, we give an application of theorem \ref{theorem:cubiccorrespondence} to finding an explicit Euler product for the co-type zeta function.
Apart from the local factors associated to the ramified primes, we can do this for any monogenic cubic order $\mathbb{Z}[\alpha]$, but for explicitness regarding these ramified primes, we only present the result for the specific example $\mathbb{Z}[2^{1/3}]$.

For an ideal $I \subset \mathbb{Z}[2^{1/3}]$ denote by $N_1(I)$, $N_2(I)$, $N_3(I)$ the invariant factors of $\mathbb{Z}[2^{1/3}] / I$, so
\begin{equation}
  \label{eq:invariantfactors}
  \mathbb{Z}[2^{\frac{1}{3}}] / I = \mathbb{Z} / N_1(I)\mathbb{Z} \oplus \mathbb{Z} / N_2(I)\mathbb{Z} \oplus \mathbb{Z} / N_2(I) \mathbb{Z}
\end{equation}
with $N_3(I) \mid N_2(I) \mid N_1(I)$.
Then the co-type zeta function for $\mathbb{Z}[2^{1/3}]$ is defined by
\begin{equation}
  \label{eq:cotypezeta}
  \zeta_{\mathbb{Z}[2^{1/3}]}(s_1, s_2, s_3) = \sum_{0\neq I \subset \mathbb{Z}[2^{1/3}]} N_1(I)^{-s_1} N_2(I)^{-s_2} N_3(I)^{-s_3}.
\end{equation}
This kind of object is of interest in the study of subgroup growth, see for example \cite{LubotzkySegal2003}.
More directly analogous to our proposition \ref{proposition:cotypezeta} is the calculation of the cotype zeta function for sublattices of $\mathbb{Z}^d$, see for example \cite{Petrogradsky2007} and \cite{ChintaKaplanKoplewitz2017}.

\begin{proposition}
  \label{proposition:cotypezeta}
  We have
  \begin{equation}
    \label{eq:cotypezeta1}
    \begin{split}
      & \zeta_{\mathbb{Z}[2^{1/3}]}(s_1, s_2, s_3) = \\
      & \quad = (1 + 2^{-s_1} + 2^{-s_1 - s_2})(1 + 3^{-s_1} + 3^{-s_1 - s_2}) \zeta(s_1 + s_2 + s_3) \\
      & \quad \quad  \times \prod_{p\in\mathcal{P}_1} \left( \frac{ 1 + 2p^{-s_1} + 2p^{-s_1 -s_2} + p^{-2s_1 -s_2}}{(1 - p^{-s_1})(1 - p^{-s_1 - s_2})} \right)\prod_{p\in\mathcal{P}_2} \left( \frac{1 - p^{-2s_1 - s_2}}{(1 - p^{-s_1})(1-p^{-s_1-s_2})}\right), \\
    \end{split}
  \end{equation}
  where $\mathcal{P}_1$ is the set of primes $p\in \mathbb{Z}$ that totally split in $\mathbb{Z}[2^{1/3}]$ and $\mathcal{P}_2$ is the set of primes that factor as degree one times degree two primes in $\mathbb{Z}[2^{1/3}]$.
\end{proposition}

We remark that we can explicitly describe $\mathcal{P}_2$ as the primes in $\mathbb{Z}$ that are $2 \pmod 3$ and $\mathcal{P}_1$ as those that are $1 \pmod 3$ and representable by the binary quadratic form $X^2 + 27Y^2$.
We also remark that (\ref{eq:cotypezeta1}) generalizes the classic factorization of the Dedekind zeta function in the quadratic case, (\ref{eq:dedekindzeta1}).
Indeed, (\ref{eq:dedekindzeta1}) can easily be modified to give an Euler product for the co-type zeta function for a quadratic order, and on the other hand (\ref{eq:cotypezeta1}) gives an interesting factorization of the Dedekind zeta function for $\mathbb{Z}[2^{1/3}]$ after setting $s_1 = s_2 = s_3$.
We note that in contrast to the quadratic setting, the Dirichlet series $\zeta_{\mathbb{Z}[2^{\frac{1}{3}}]}(s, s, s) / \zeta(3s)$ does not count the roots of the cubic congruence but rather pairs of roots as in theorem \ref{theorem:cubiccorrespondence}. 

Our final application is to the composition of ideals, reflecting the fact that the arithmetic connection between roots of congruence and ideals goes deeper than just the counting illustrated in proposition \ref{proposition:cotypezeta}.
Using theorem \ref{theorem:correspondence} we obtain the following theorem.

\begin{proposition}
  \label{proposition:composition}
  Let $I$ and $J$ be unramified ideals in $\mathbb{Z}[\alpha]$, i.e. coprime to the discriminant, such that $\mathbb{Z}[\alpha]/ I$ and $\mathbb{Z}[\alpha] / J$ are additively cyclic.
  Let $\mu \pmod m$ and $\nu \pmod n$ be the roots of the congruence corresponding to $I$ and $J$ via theorem \ref{theorem:correspondence}.
  Then $\mathbb{Z}[\alpha] / IJ$ is cyclic if and only if $\mu \equiv \nu \pmod {\gcd(m,n)}$, and in this is the case $IJ$ corresponds to the unique root $\tilde{\mu} \pmod {mn}$ satisfying $\tilde{\mu} \equiv \mu \pmod {m}$ and $\tilde{\mu} \equiv \nu \pmod n$. 
\end{proposition}

We also note that for a degree one prime $P$ corresponding to a root $\mu \pmod p$, then all the conjugate primes to $P$ correspond to the different roots modulo $p$.
In the quadratic case, this means that the conjugate $P'$ corresponds to the root $-a_1 - \mu_1$, and so we note that multiplying the ideals $P$, $P'$ gives the rational integer $p$.
Using this fact together with proposition \ref{proposition:composition} gives a full account of how one can understand composition of ideals in quadratic orders in terms of roots of congruences.

Clearly this breaks down in higher degree since the product of two conjugate, distinct degree one primes is no longer a rational integer.
This is the same phenomenon that motivated our extension of theorem \ref{theorem:correspondence} to theorem \ref{theorem:cubiccorrespondence}, and in the same way we can ask if there is an extension of proposition \ref{proposition:composition} that gives a more robust understanding of composition of ideals in terms of roots of congruence, at least in the cubic setting.
Proposition \ref{proposition:cubiccomposition} provides such an extension.

\begin{proposition}
  \label{proposition:cubiccomposition}
  Let $d = 3$, let $I$ and $J$ be unramified ideals in $\mathbb{Z}[\alpha]$ that are not divisible by rational integers, and let $\mu_1 \pmod {m_1}$, $\mu_2\pmod {m_2}$ and $\nu_1 \pmod {n_1}$, $\nu_2 \pmod{n_2}$ be the roots of the congruence corresponding to $I$ and $J$ via theorem \ref{theorem:cubiccorrespondence}.
  Then $IJ$ is not divisible by rational integers if and only if $\mu_1 \equiv \nu_1 \pmod {\gcd(m_1,n_1)}$, $\gcd(m_1, n_2, \mu_1 - \nu_2) = 1$, and $\gcd(m_2, n_1, \mu_2 - \nu_1) = 1$.
  Moreover, if $IJ$ is not divisible by rational integers, then, setting $l = \frac{\gcd(m_2,n_2)}{\gcd(m_2,n_2, \mu_2 - \nu_2)}$, $IJ$ corresponds to the roots $\tilde{\mu}_1 \pmod {m_1n_1 l}$ and $\tilde{\mu}_2 \pmod {\frac{m_2n_2}{l^2}}$ satisfying $\tilde{\mu}_1 \equiv \mu_1 \pmod{m_1}$, $\tilde{\mu}_1 \equiv \nu_1 \pmod {n_1}$, $\gcd(l, (\tilde{\mu}_1 - \mu_2)(\tilde{\mu}_1 - \nu_2)) = 1$, and $\tilde{\mu}_2 \equiv \mu_2 \pmod {\frac{m_2}{l}}$, $\tilde{\mu}_2 \equiv \nu_2 \pmod {\frac{n_2}{l}}$.
\end{proposition}

We illustrate this proposition by considering some examples.
First we note that if ideals $I$ and $J$ have coprime norm, then proposition \ref{proposition:cubiccomposition} simply states that the product ideal $IJ$ corresponds to the roots obtained from the corresponding to $I$ and $J$ by the Chinese remainder theorem.
We remark that verifying this is in fact the first step towards proving proposition \ref{proposition:cubiccomposition}, see lemma \ref{lemma:coprimecubiccomposition}.

Now suppose that $N(I)$ and $N(J)$ are powers of the same prime $p$, $p$ not dividing the discriminant of $F$.
Then $p$ either remains prime in $\mathbb{Z}[\alpha]$, factors as a degree one times a degree two prime ideal, or factors as a product of degree one prime ideals.
As implied by either theorem \ref{theorem:cubiccorrespondence} or the Dedekind-Kummer theorems, these cases correspond exactly to whether $F(\mu) \equiv 0 \pmod p$ has zero, one, or three solutions.
The case when $p$ factors in $\mathbb{Z}[\alpha]$ as a degree one prime $P_1$ times a degree two prime $P_2$, bases for $P_1$ and $P_2$ are given respectively by
\begin{equation}
  \label{eq:degreetwobases}
  \begin{pmatrix}
    1 & a_1 & \mu^2 - a_1\mu \\
    0 & 1 & -\mu \\
    0 & 0 & p
  \end{pmatrix}
  \begin{pmatrix}
    \alpha^2 \\
    \alpha \\
    1
  \end{pmatrix}
  , \quad
  \begin{pmatrix}
    1 & \mu + a_1 & \mu^2 + a_1\mu + a_2 \\
    0 & p & 0 \\
    0 & 0 & p
  \end{pmatrix}
  \begin{pmatrix}
    \alpha^2 \\
    \alpha \\
    1
  \end{pmatrix}
  ,
\end{equation}
where $\mu \pmod m$ is the unique solution to $F(\mu) \equiv 0 \pmod m$.
From proposition \ref{proposition:composition} it follows that powers of $P_1$ correspond to lifting $\mu$ to roots modulo powers of $p$, and proposition \ref{proposition:cubiccomposition} implies that the same holds for powers of $P_2$.

In the case when $p$ factors as the product of three degree one primes $P_1, P_2, P_3$, we write the basis for $P_j$ as
\begin{equation}
  \label{eq:degreeonebases}
  \begin{pmatrix}
    1 & a_1 & -\mu_j^2 - a_1\mu_j \\
    0 & 1 & -\mu_j \\
    0 & 0 & p
  \end{pmatrix}
  ,
\end{equation}
where $\mu_j\pmod p$ is one of the three solutions to $F(\mu) \equiv 0 \pmod p$.
As before, it follows from proposition \ref{proposition:composition} that taking powers of $P_j$ corresponds to lifting $\mu_j$ to roots modulo powers of $p$.
More interesting is that proposition \ref{proposition:cubiccomposition} implies that $P_2P_3$ has basis
\begin{equation}
  \label{eq:P2P3basis}
  \begin{pmatrix}
    1 & \mu_1 + a_1 & \mu_1^2 + a_1\mu_1 + a_2 \\
    0 & p & 0 \\
    0 & 0 & p
  \end{pmatrix}
  \begin{pmatrix}
    \alpha^2 \\
    \alpha \\
    1
  \end{pmatrix}
  ,
\end{equation}
and we note that in the context of proposition \ref{proposition:cubiccomposition}, this is consistent with $P_1P_2P_3$ being divisible by a rational integer.
It is also interesting to consider $P_2^2P_3$, which has a basis of the form
\begin{equation}
  \label{eq:P22P3basis}
  \begin{pmatrix}
    1 & \mu_1 + a_1 & * \\
    0 & p & -\mu_2 p \\
    0 & 0 & p^2
  \end{pmatrix}
  \begin{pmatrix}
    \alpha^2 \\
    \alpha \\
    1
  \end{pmatrix}
  .
\end{equation}
One can understand proposition \ref{proposition:cubiccomposition} in this setting as stating that, when $l\leq k$, $P_2^k P_3^l = P_2^{k-l}(P_2P_3)^l$ corresponds to $\mu_1$ lifted to a root modulo $p^l$ and $\mu_2$ lifted to a root modulo $p^{k-l}$.
Verifying this is in fact a key step towards proving proposition \ref{proposition:cubiccomposition}, see lemma \ref{lemma:primecubiccomposition}.

\subsection*{Acknowledgements}
\label{acknowledgements}

The author completed much of the work presented here for his Ph.D. thesis, and thus under the supervision of his advisor, Henryk Iwaniec.
He would like to express his sincere gratitude to Henryk Iwaniec and to the rest of his committee members, Alex Kontorovich, Steven D. Miller, and Nigel J.E. Pitt.
The author would also like to thank Jens Marklof for his encouragement to finish this paper and the referee for the careful reading and many helpful suggestions.

\section{Correspondence between roots and ideals}
\label{sec:correspondence}

In this section we prove theorems \ref{theorem:correspondence} and \ref{theorem:cubiccorrespondence}.
We start with a sublattice $I$ of the ring $\mathbb{Z}[\alpha]$, which has a unique basis $\{\beta_1, \dots, \beta_d\}$ in Hermite normal form.
This is to say that
\begin{equation}
  \label{eq:hnfbasis}
  \begin{pmatrix}
    \beta_1 \\
    \vdots \\
    \beta_{d-1} \\
    \beta_d 
  \end{pmatrix} = B
  \begin{pmatrix}
    \alpha^{d-1} \\
    \vdots \\
    \alpha \\
    1
  \end{pmatrix}
\end{equation}
where $ B = (b_{ij})$ is upper-triangular, i.e. $b_{ij} = 0$ if $i > j$, and the entries satisfy $b_{jj} > 0$ and $0 \leq b_{ij} < b_{jj}$ for all $i< j$.

Since $\alpha$ generates the ring $\mathbb{Z}[\alpha]$, the lattice $I$ is an ideal if and only if $\alpha I \subset I$.
Hence $I$ is an ideal if and only if the matrix is $B A B^{-1}$ is integral, where $A$ is the matrix by which $\alpha$ acts on $\mathbb{Z}[\alpha]$ with respect to the basis $\{\alpha^{d-1}, \dots, 1\}$.
Explicitly, using the coefficients from (\ref{eq:Fdef}), we have $A = (a_{ij})$ with
\begin{equation}
  \label{eq:aijdef}
  a_{ij} =
  \begin{cases}
     - a_j & \mathrm{if\ } i = 1 \\
    1 & \mathrm{if\ } 2\leq i \leq d, j = i-1 \\
    0 & \mathrm{otherwise},
  \end{cases}
\end{equation}

Letting $B^{-1} = (b_{ij}')$, we observe that for $2\leq i \leq d$, the $(i,j)$ entry of $BAB^{-1}$ is
\begin{equation}
  \label{eq:ijentry}
  \sum_{i\leq k \leq j+1} b_{ik}b_{(k-1)j}',
\end{equation}
where we set $b_{i(d+1)} = 0$ for convenience.
In particular, the $(i,j)$ entry is $0$ if $j\leq i-2$ and the $(i, i-1)$ entry is simply $b_{ii}b_{(i-1)(i-1)}' = \frac{b_{ii}}{b_{(i-1)(i-1)}}$.
From this we see that for $I$ to be an ideal, it is necessary to have
\begin{equation}
  \label{eq:midef}
  b_{dd} = m_db_{(d-1)(d-1)} = m_dm_{d-1}b_{(d-2)(d-2)} = \cdots = \prod_{1\leq i \leq d} m_{i}. 
\end{equation}
The diagonal entries of $BAB^{-1}$ are not much more difficult to compute due to the fact that $b_{(i-1)i}' = - \frac{b_{(i-1)i)}}{b_{ii}b_{(i-1)(i-1)}}$.
With this, the $(i,i)$ entry, $2\leq i \leq d$, is
\begin{equation}
  \label{eq:diag}
  \frac{b_{i(i+1)}}{b_{ii}} - \frac{b_{(i-1)i}}{b_{(i-1)(i-1)}}. 
\end{equation}
Applied with $i=d$, we see that for $I$ to be an ideal, it is necessary that $b_{(d-1)d} = c_{(d-1)d}b_{(d-1)(d-1)}$ for some integer $c_{(d-1)d}$.
Continuing inductively with $i = d-1, d-2, \dots, 2$, we see that for all $2\leq i \leq d$, $b_{(i-1)i} = c_{(i-1)i} b_{(i-1)(i-1)}$ is necessary.

As mentioned in the introduction, continuing this analysis, trying to obtain exactly the necessary and sufficient conditions for $I$ to be an ideal, is a bit unwieldy here in such a general setting.
We instead focus on a special case in which the calculations simplify significantly, and proposition \ref{proposition:invariantfactors} proved below characterizes this special case that we consider.
We however do succeed without such simplifications in the cubic case, $d = 3$, resulting theorem \ref{theorem:cubiccorrespondence}, which we prove in section \ref{sec:cubiccorrespondence} below. 

\subsection{Proof of proposition \ref{proposition:invariantfactors}}
\label{sec:generalinvariantfactors}

Our method is to prove that $b_{(i-j)(i-j)}$ divides $b_{(i-j)i}$ by first inducting on $i = d, d-1, \dots, j+1$, in this order, and then on $j$; we have already handled the case $j =1$ and arbitrary $i$ above.
Let $j > 1$ and assume the divisibility condition for all smaller $j$ and arbitrary $i$.
The base case for inducting on $i$ is $i=d$, and to prove the divisibility here we consider the $(d-j + 1,d)$ entry of $BAB^{-1}$, which is
  \begin{equation}
    \label{eq:basecase}
    \sum_{d-j+1 \leq k \leq d+1} b_{(d-j+1)k}b_{(k-1)d}'. 
  \end{equation}
Since
\begin{equation}
  \label{eq:bprime1}
  b_{(k-1)d}' = (-1)^{d + k +1} \prod_{k-1 \leq l \leq d} b_{ll}^{-1} \det (b_{rs})_{\substack{k-1 \leq r \leq d -1 \\ \max\{k,r\} \leq s \leq d}},
\end{equation}
using the convention that the determinant of a $0\times 0$ matrix is $1$, we apply the induction hypothesis to the $b_{rs}$ to see that for $k > d-j +1$, $b_{(k-1)d}'$ is a fraction with denominator $b_{dd}$.
For $k=d - j +1$, we perform a co-factor expansion along the top row of the determinant in (\ref{eq:bprime1}), noting that for $s < d$ we can apply the induction hypothesis to see that
\begin{equation}
  \label{eq:bprime2}
  b_{(d-j)d}' = \frac{\mathrm{integer}}{b_{dd}} \pm \frac{b_{(d-j)d}}{b_{dd}b_{(d-j)(d-j)}}. 
\end{equation}
Note that we have used the fact that $B$ is upper-triangular to compute the co-factor associated with $r = d-j+1$, $s = d$. 
Putting these facts into (\ref{eq:basecase}) and applying both the inductive hypothesis for $k<d+1$ to write $b_{(d-j+1)k} = c_{(d -j +1)k}b_{(d-j+1)(d-j+1)}$ and also the previously noted $b_{dd} = b_{(d-j+1)(d-j+1)}\prod_{1\leq l \leq j-1}m_{d-l+1}$, we see that the $(d-j+1,d)$ entry of $BAB^{-1}$ has the form
\begin{equation}
  \label{eq:j1entry}
  \frac{\mathrm{integer}}{\prod_{1\leq l \leq j -1}m_{d-l+1}} \pm \frac{b_{(d-j)d}}{b_{(d-j)(d-j)}\prod_{l=1}^{j-1}m_{d-l+1}}.
\end{equation}
From this it is clearly necessary for $b_{(d-j)(d-j)}$ to divide $b_{(d-j)d}$, thus proving the base case for this induction.

The general case for the induction on $i = d-1, \dots, j+1$ follows similarly.
We consider now the $(i-j+1, i)$ entry of $BAB^{-1}$, which is
\begin{equation}
  \label{eq:induction}
  \sum_{i-j+1 \leq k \leq i+1} b_{(i-j+1)k}b_{(k-1)i}'. 
\end{equation}
Here we have
\begin{equation}
  \label{eq:bprime3}
  b_{(k-1)i}' = (-1)^{i+k+1} \prod_{k-1\leq l \leq i} b_{ll}^{-1} \det (b_{rs})_{\substack{k-1 \leq r \leq i-1 \\ \max\{k,r\}\leq s \leq i}},
\end{equation}
where we again use the convention that the determinant of a $0\times 0$ matrix is $1$.
For all except the first term of (\ref{eq:induction}), i.e. $i-j +1 < k \leq i+1$, we can apply the $j$ inductive hypothesis to see that each $b_{(k+1)(i-j)}'$ is a fraction with denominator $b_{ii}$.
And as before, for $k = i-j + 1$, we perform a co-factor expansion of the determinant in (\ref{eq:bprime3}) along the top row, applying the $j$ induction hypothesis for $s < i$, to see that
\begin{equation}
  \label{eq:bprime4}
  b_{(i-j)i}' = \frac{\mathrm{integer}}{b_{ii}} \pm \frac{b_{(i-j)i}}{b_{ii}b_{(i-j)(i-j)}}.
\end{equation}
To use these expressions in (\ref{eq:induction}), we note that for $i-j+1 \leq k \leq i$, i.e. all be the last term, we apply the $j$ induction hypothesis to see that $b_{(i-j+1)k} = c_{(i-j+1)k}b_{(i-j+1)(i-j+1)}$.
And for the last term $k = i+1$, we apply the $i$ induction hypothesis to see the same.
Using that $b_{ii} = b_{(i-j+1)(i-j+1)}\prod_{1\leq l \leq j-1} m_{i - l + 1}$, we now see that the $(i-j+1,i)$ entry of $BAB^{-1}$ has the form
\begin{equation}
  \label{eq:?entry}
  \frac{\mathrm{integer}}{\prod_{1\leq l \leq j-1} m_{i - l + 1}} \pm \frac{b_{(i-j)i}}{b_{(i-j)(i-j)}\prod_{1\leq l \leq j-1} m_{i - l + 1}}.
\end{equation}
This clearly shows that it is necessary for $b_{(i-j)(i-j)}$ to divide $b_{(i-j)i}$, thus finishing the induction.

To finish the proof of proposition \ref{proposition:invariantfactors} we observe that the divisibility conditions on the entries of $B$ show that $B$ can be brought into diagonal form, with each diagonal entry dividing the next, simply by multiplying on the right by an upper-triangular matrix in $\mathrm{SL}(d, \mathbb{Z})$.
This shows that the diagonal entries of $B$ are also the diagonal entries of the basis of the ideal $I$ when written in Smith normal form.
Hence the diagonal entries of $B$ are the invariant factors of $\mathbb{Z}[\alpha] / I$ as required.

\subsection{Proof of theorem \ref{theorem:correspondence}}
\label{sec:correspondenceproof}

As mentioned previously, we now make an assumption on the lattice $I$ in order to simplify calculations.
The assumption we make is that the quotient $\mathbb{Z}[\alpha] / I$ is (additively) cyclic.
If $I$ is an ideal, this assumption, via proposition \ref{proposition:invariantfactors}, implies that $m_j = 1$ except for $j=d$, where the $m_j$ are defined by (\ref{eq:midef}).
Let us set $m_d = m$.
By our assumption that $B$ is in Hermite normal form, specifically that $0\leq b_{ij} < b_{jj}$ for all $ i < j$, we also see that the cyclicity assumption implies that all the off-diagonal entries in $B$ are $0$ outside of the last column.

Having this assumption, we observe first that
\begin{equation}
  \label{eq:bprime5}
  b_{ij}' =
  \begin{cases}
    1 & \mathrm{if\ } i = j < d \\
    \frac{1}{m} & \mathrm{if\ } i = j = d \\
    -\frac{b_{id}}{m} & \mathrm{if\ } i< d, j=d \\
    0 & \mathrm{otherwise}.
  \end{cases}
\end{equation}
Now, for $2 \leq i \leq d$, the $(i,j)$ entry of $BAB^{-1}$ will be
\begin{equation}
  \label{eq:ijentry2}
  \begin{cases}
    - b_{(d-1)d} & \mathrm{if\ } i=j=d \\
    -\frac{1}{m}(b_{id}b_{(d-1)d} + b_{(i-1)d}) & \mathrm{if\ } i < d, j=d \\
    b_{id} & \mathrm{if\ } j=d-1 \\
    1 & \mathrm{if\ } i = j + 1 < d \\
    0 & \mathrm{otherwise}.
  \end{cases}
\end{equation}
Only the second case of the above gives an integrality condition, which, setting $b_{(d-1)d} = -\mu$, is satisfied if and only if
\begin{equation}
  \label{eq:bi1def}
  b_{id} \equiv -\mu^{d-i} \pmod m
\end{equation}
for all $1 \leq i \leq d-1$.

It remains to analyze the integrality conditions arising from the top row of $BAB^{-1}$.
A relatively quick calculation shows that the $(1,j)$ entry of $BAB^{-1}$ is
\begin{equation}
  \label{eq:1jentry}
  \sum_{1\leq l \leq j} a_{1l} b_{lj}' + b_{1d}b_{(d-1)j}' \equiv   
  \begin{cases}
    -a_j & \mathrm{if\ } j < d - 1 \\
    -a_{d-1} - \mu^{d-1} & \mathrm{if\ } j = d-1 \\
    - \frac{1}{m}\left(\mu^d + \sum_{1\leq l\leq d} a_l \mu^{d - l}\right) & \mathrm{if\ } j =d,
  \end{cases}
\end{equation}
modulo $1$, recalling (\ref{eq:Fdef}), (\ref{eq:aijdef}), (\ref{eq:bprime5}), and (\ref{eq:bi1def}).
Hence the integrality condition on $BAB^{-1}$ is satisfied if and only if $\mu$ is a root of the polynomial congruence
\begin{equation}
  \label{eq:polycongruence}
  \mu^d  +  a_1 \mu^{d-1} + \cdots + a_d \equiv 0 \pmod m, 
\end{equation}
and theorem \ref{theorem:correspondence} follows.

\subsection{Proof of theorem \ref{theorem:cubiccorrespondence}}
\label{sec:cubiccorrespondence}

In view of proposition \ref{proposition:invariantfactors}, we see that the entry $b_{11}$ of $B$ divides all other entries of $B$, $B$ as in (\ref{eq:hnfbasis}).
It follows that when considering the conditions for $BAB^{-1}$ to be an integral matrix, $b_{11}$ does not make an appearance.
We therefore may as well assume $b_{11} = 1$, which corresponds to assuming that $I$ is not divisible by any rational integers, as indicated in the introduction.
In the case $d = 3$ we therefore write the matrix $B$ as
\begin{equation}
  \label{eq:hnfbasis1}
  B =
  \begin{pmatrix}
    1 & \mu_1 + a_1 & \lambda \\
    0 & m_1 & -\mu_2 m_1 \\
    0 & 0 & m_1m_2
  \end{pmatrix}
  .
\end{equation}
Hence $BAB^{-1}$ is given by
\begin{equation}
  \label{eq:BABinverse}
  \begin{pmatrix}
    \mu_1 & \frac{1}{m_1}(-\mu_1^2 - a_1\mu_1 - a_2 + \lambda) & \frac{1}{m_1m_2}( -\mu_1^2 \mu_2 - a_1\mu_1\mu_2 - a_2\mu_2 - a_3 + (\mu_2 - \mu_1)\lambda) \\
    m_1 & -\mu_1 - \mu_2 - a_1 & \frac{1}{m_2} (-\mu_1\mu_2 - a_1 \mu_2 - \mu_2^2 - \lambda) \\
    0 & m_2 & \mu_2
  \end{pmatrix}
  .
\end{equation}

From the $(1,2)$ and $(2,3)$ entries of (\ref{eq:BABinverse}), we see that for $I$ to be an ideal, it is necessary that $\lambda$ satisfy the congruences
\begin{equation}
  \label{eq:lambdacongruence5}
  \begin{split}
    \lambda & \equiv \mu_1^2 + a_1 \mu_1 + a_2 \pmod {m_1} \\
    \lambda & \equiv -\mu_2^2 - \mu_1\mu_2 - a_1\mu_2 \pmod {m_2}.
  \end{split}
\end{equation}
In order for the congruences (\ref{eq:lambdacongruence5})to have a solution, it is necessary that
\begin{equation}
  \label{eq:gcdcongruence}
  \mu_1^2 + \mu_1\mu_2 + \mu_2^2 + a_1 \mu_1 + a_1 \mu_2 + a_2 \equiv 0 \pmod {\gcd (m_1, m_2)}.
\end{equation}
Assuming that (\ref{eq:gcdcongruence}) holds, $\lambda$ will be given by
\begin{flalign}
  \label{eq:lambdadef}
  \lambda & \equiv (\mu_1^2 + a_1\mu_1 + a_2) \frac{\overline{m}_2 m_2}{\gcd(m_1,m_2)} - (\mu_2^2 + \mu_1\mu_2 + a_1\mu_2) \frac{\overline{m}_1 m_1}{\gcd(m_1, m_2)} \\ \nonumber
  & \quad + \kappa \frac{m_1m_2}{\gcd(m_1,m_2)} \pmod {m_1m_2}, 
\end{flalign}
where $\kappa \pmod {\gcd(m_1, m_2)}$ is to be determined and $\overline{m}_1$, $\overline{m}_2$ are defined by
\begin{equation}
  \label{eq:mbardef}
  \frac{\overline{m}_1m_1}{\gcd(m_1,m_2)} + \frac{\overline{m}_2 m_2}{\gcd(m_1,m_2)} = 1.
\end{equation}
We note that such $\overline{m}_1$ and $\overline{m}_2$ are not unique, but all solutions to (\ref{eq:mbardef}) can be obtained by a given solution $\overline{m}_1$, $\overline{m}_2$ by respectively adding, subtracting $\frac{lm_2}{\gcd (m_1,m_2)}$, $\frac{lm_1}{\gcd(m_1,m_2)}$, where $l$ is an arbitrary integer.
We note that in (\ref{eq:lambdadef}), applying these changes to $\overline{m}_1$, $\overline{m}_2$ has the effect of replacing $\kappa$ with
\begin{equation}
  \label{eq:kappachange}
  \kappa - l \frac{\mu_1^2 + \mu_1\mu_2 + \mu_2^2 + a_1\mu_1 + a_1 \mu_2 + a_2}{\gcd( m_1, m_2)}.
\end{equation}

Putting (\ref{eq:lambdadef}) into the $(1,3)$ entry of (\ref{eq:BABinverse}) and making use of (\ref{eq:mbardef}), we find that for $I$ to be an ideal, it is necessary that
\begin{flalign}
  \label{eq:m1m2congruence}
  & (\mu_1^3 + a_1\mu_1^2 + a_2\mu_1 + a_3) \frac{\overline{m}_2m_2}{\gcd(m_1,m_2)} + (\mu_2^3 + a_1 \mu_2^2 + a_2\mu_2 + a_3 ) \frac{\overline{m}_1m_1}{\gcd(m_1,m_2)} \\ \nonumber
  & \quad + (\mu_1 - \mu_2) \kappa \frac{m_1m_2}{\gcd(m_1, m_2)} \equiv 0 \pmod {m_1m_2}. 
\end{flalign}
Using (\ref{eq:mbardef}), we write the left side of (\ref{eq:m1m2congruence}) either as
\begin{equation}
  \label{eq:m1congruence}
  \mu_1^3 + a_1\mu_1^2 + a_2\mu_1 + a_3 + (\mu_1 - \mu_2)\left( \kappa - \frac{\mu_1^2 + \mu_1\mu_2 + \mu_2^2 + a_1\mu_1 + a_1 \mu_2 + a_2}{\gcd( m_1, m_2)}\right) \overline{m}_1 m_1
\end{equation}
or
\begin{equation}
  \label{eq:m2congruence}
  \mu_2^3 + a_1 \mu_2^2 + a_2\mu_2 + a_3 + (\mu_1 - \mu_2)\left( \kappa + \frac{\mu_1^2 + \mu_1\mu_2 + \mu_2^2 + a_1\mu_1 + a_1 \mu_2 + a_2}{\gcd( m_1, m_2)}\right) \overline{m}_2 m_2.
\end{equation}
We see that (\ref{eq:m1congruence}), (\ref{eq:m2congruence}) respectively imply that
\begin{equation}
  \label{eq:mu1mu2congruence}
  F(\mu_1) \equiv 0 \pmod {m_1}, \quad F(\mu_2)\equiv 0 \pmod {m_2}
\end{equation}
are both necessary for $I$ to be an ideal.

Assuming the validity of (\ref{eq:mu1mu2congruence}), we multiply (\ref{eq:m1m2congruence}) by $\frac{\gcd(m_1, m_2)}{m_1m_2}$ to obtain the necessary condition
\begin{equation}
  \label{eq:gcdcongruence1}
  \frac{F(\mu_1)}{m_1} \overline{m}_2 + \frac{F(\mu_2)}{m_2} \overline{m}_1 + (\mu_1 - \mu_2)\kappa \equiv 0 \pmod {\gcd(m_1,m_2)}. 
\end{equation}
We note that different choices of $\overline{m}_1$, $\overline{m}_2$ leave the condition (\ref{eq:gcdcongruence1}) invariant after applying the corresponding change to $\kappa$, (\ref{eq:kappachange}).

We observe that if $\gcd(m_1, m_2, \mu_1 - \mu_2) = 1$, then there will be a unique $\kappa$ satisfying (\ref{eq:gcdcongruence1}).
To address the converse of this statement, we note that if there is a prime $p$ dividing all of $m_1$, $m_2$, and $\mu_1 - \mu_2$, then 
\begin{equation}
  \label{eq:pcongruence}
  \begin{split}
    \mu_1 & \equiv \mu_2 \pmod p, \\
    F(\mu_1) & \equiv 0 \pmod p, \\
    0 & \equiv \mu_1^2 + \mu_1\mu_2 + \mu_2^2 + a_1 \mu_1 + a_1 \mu_2 + a_2 \equiv F'(\mu_1) \pmod p.
  \end{split}
\end{equation}
The second and third lines of (\ref{eq:pcongruence}) show that in this case $p$ must divide the discriminant of $F$.

This is enough to prove the first part of theorem \ref{theorem:cubiccorrespondence}.
To show the converse part, we assume that $\mu_1 \pmod {m_1}$ and $\mu_2 \pmod {m_2}$ satisfy (\ref{eq:mu1mu2congruence}) and $\gcd(m_1, m_2, \mu_1 - \mu_2) = 1$.
We have
\begin{equation}
  \label{eq:gcdcongruence2}
  0 \equiv F(\mu_1) - F(\mu_2) = (\mu_1 - \mu_2)(\mu_1^2 + \mu_1\mu_2 + \mu_2^2 + a_1\mu_1 + a_1\mu_2 + a_2) \pmod {\gcd(m_1, m_2)},
\end{equation}
and so $\gcd(m_1, m_2, \mu_1 - \mu_2) = 1$ implies that $\mu_1$, $\mu_2$ satisfy (\ref{eq:gcdcongruence}).
It then follows that (\ref{eq:gcdcongruence1}) and (\ref{eq:lambdadef}) give a uniquely defined $\lambda \pmod {m_1m_2}$ such that the matrix (\ref{eq:BABinverse}) has integral entries, thus proving the converse part of theorem \ref{theorem:cubiccorrespondence}.

\section{Parameterizing the roots}
\label{sec:parameterization}

We now turn to the proofs of the theorems parameterizing the roots of the polynomial congruence, theorems \ref{theorem:parameterization}, \ref{theorem:parameterization1}, \ref{theorem:cubicparameterization}, and \ref{theorem:cubicparameterization1}. 
These proofs are all incarnations of the same idea: that after fixing representatives of the narrow class group, as done in the introduction, every invertible ideal inherits a natural basis.
This basis is obtained in the following way.
Suppose that an invertible ideal $I \subset \mathbb{Z}[\alpha]$ is in the ideal class represented by the integral ideal $I_l$, so there is an element $\xi \in I_l^{-1}$ such that $I = \xi I_l$.
After fixing a basis $\{\beta_{l1}, \dots, \beta_{ld}\}$ of $I_l$ as in the introduction, we see that $\{\xi \beta_{l1}, \dots, \xi\beta_{ld}\}$ forms a natural basis for $I$.

Theorems \ref{theorem:parameterization} and \ref{theorem:cubicparameterization}, or at least the first part of these theorems, are almost proved once one makes the observation that given two bases for an ideal $I$, the one from the previous paragraph and the one from theorems \ref{theorem:correspondence} and \ref{theorem:cubiccorrespondence}, then they must be related by a matrix $\gamma \in \mathrm{GL}(d,\mathbb{Z})$.
The remaining details and the other parts of these theorems are verified in sections \ref{sec:parameterization1} and \ref{sec:cubicparameterization1} below.
The proofs of theorems \ref{theorem:parameterization1} and \ref{theorem:cubicparameterization1} are not much more difficult.
The main work, which is also carried out in sections \ref{sec:parameterization1} and \ref{sec:cubicparameterization1}, is to explicitly compute this change-of-basis matrix $\gamma$ in terms of the $\xi$ referred to in the previous paragraph.

\subsection{Proof of theorems \ref{theorem:parameterization} and \ref{theorem:parameterization1}}
\label{sec:parameterization1}

We begin with the proof of theorem \ref{theorem:parameterization}.
As mentioned in the statement of this theorem, we only consider the roots $\mu \pmod m$ that correspond, via theorem \ref{theorem:correspondence}, to invertible ideals.
It would be nice to have a more concrete characterization of the $\mu \pmod m$ in terms of $m$ and $\mu$ themselves, but we unfortunately have not yet been able to do this in general.
We note however that it is sufficient to have $m$ coprime to the discriminant $D$ of $F$. 

Let $\mu \pmod m$ correspond to such an ideal, say $I$, and let $I_l$ be the fixed representative of the narrow ideal class of $I$.
We then have that $I = \xi I_l$ for some $\xi \in I_l^{-1}$ and $\xi \in C_{r_1,r_2}^+$, as we are considering here the narrow classes.
As mentioned in the introduction to this section, this $\xi$ gives a natural basis for $I$, namely $\{\xi\beta_{l1}, \dots, \xi\beta_{ld}\}$.
Written in terms of the embeddings, i.e. as a vector in $C_{r_1,r_2}$, we have
\begin{equation}
  \label{eq:parameterization3}
  \begin{pmatrix}
    \xi \beta_{l1} \\
    \vdots \\
    \xi \beta_{ld}
  \end{pmatrix}
  = \mathfrak{B}_l
  \begin{pmatrix}
    \xi^{(1)} & \cdots & 0 \\
    \vdots & \ddots & \vdots \\
    0 & \cdots & \xi^{(d)}
  \end{pmatrix}
  .
\end{equation}

We now have two bases for the ideal $I$, the basis (\ref{eq:parameterization3}) and the basis from theorem \ref{theorem:correspondence}.
These bases must be related by an element of $\mathrm{GL}(d,\mathbb{Z})$, and hence there is $\gamma \in \mathrm{GL}(d,\mathbb{Z})$ such that
\begin{equation}
  \label{eq:changebasis}
  \begin{pmatrix}
    1 & \cdots & 0 & -\mu^{d-1} \\
    \vdots & \ddots & \vdots & \vdots \\
    0 & \cdots & 1 & -\mu \\
    0 & \cdots & 0 & m 
  \end{pmatrix}
  \mathfrak{B}_1 = \gamma \mathfrak{B}_l
  \begin{pmatrix}
    \xi^{(1)} & \cdots & 0 \\
    \vdots & \ddots & \vdots \\
    0 & \cdots & \xi^{(d)}
  \end{pmatrix}
  .
\end{equation}
In fact, recalling our stipulations that $m > 0$, $\mathrm{sign} \det \mathfrak{B}_l = \mathrm{sign}\det \mathfrak{B}_1$ and $\xi \in C_{r_1,r_2}^+$, we have $\gamma \in \Gamma = \mathrm{SL}(d,\mathbb{Z})$.

To finish the first part of theorem \ref{theorem:parameterization}, we only need to verify the uniqueness of this $\gamma$ modulo the left action of $U$ and the right action of $\Gamma_l$.
Regarding the action of $u\in U$, we see from (\ref{eq:changebasis}) that replacing $\gamma$ by $u\gamma$ corresponds to adding multiples of $m$ to the $\mu^j$.
Hence different representatives $\gamma $ of $U\gamma \in U \backslash \Gamma$ simply correspond to different representatives of the $\mu^j$ modulo $m$.

Regarding the action of $\gamma_1 \in \Gamma_l$, we see that since
\begin{equation}
  \label{eq:xi1def}
  \gamma_1 = \mathfrak{B}_l
  \begin{pmatrix}
    \xi_1^{(1)} & \cdots & 0 \\
    \vdots & \ddots & \vdots \\
    0 & \cdots & \xi_1^{(d)}
  \end{pmatrix}
  \mathfrak{B}_l^{-1}, \quad \xi_1 \in C_{r_1,r_2}^+,
\end{equation}
with $\xi_1$ a unit in $\mathbb{Z}[\alpha]$, replacing $\gamma$ with $\gamma\gamma_1$ corresponds to replacing $\xi$ by $\xi \xi_1^{-1}$.
It follows that different representatives $\gamma $ of $\gamma \Gamma_l \in \Gamma / \Gamma_l$ correspond to the same $m$ and $\mu \pmod m$, just different $\xi \in C_{r_1,r_2}^+$.

We turn now to the converse part of theorem \ref{theorem:parameterization}.
We suppose that we have $\gamma \in \Gamma$ and $\xi \in C_{r_1,r_2}^+$ such that
\begin{equation}
  \label{eq:converseparameterization2}
  \gamma \mathfrak{B}_l
  \begin{pmatrix}
    \xi_1^{(1)} & \cdots & 0 \\
    \vdots & \ddots & \vdots \\
    0 & \cdots & \xi_1^{(d)}
  \end{pmatrix}
  =
  \begin{pmatrix}
    1 & \cdots & 0 & * \\
    \vdots & \ddots & \vdots & \vdots \\
    0 & \cdots & 1 & * \\
    0 & \cdots & 0 & *
  \end{pmatrix}
  \mathfrak{B}_1.
\end{equation}
We first observe that if $\xi \in I_l^{-1}$, then the left side of (\ref{eq:converseparameterization2}) would be basis for an ideal $I \subset \mathbb{Z}[\alpha]$.
That this ideal $I$ has a basis with the form on the right of (\ref{eq:converseparameterization2}) implies via proposition \ref{proposition:invariantfactors} that the quotient $\mathbb{Z}[\alpha] / I$ is cyclic.
It then follows from theorem \ref{theorem:correspondence} that in fact
\begin{equation}
  \label{eq:converseparameterization3}
  \gamma \mathfrak{B}_l
  \begin{pmatrix}
    \xi_1^{(1)} & \cdots & 0 \\
    \vdots & \ddots & \vdots \\
    0 & \cdots & \xi_1^{(d)}
  \end{pmatrix}
  =
  \begin{pmatrix}
    1 & \cdots & 0 & -\mu^{d-1} \\
    \vdots & \ddots & \vdots & \vdots \\
    0 & \cdots & 1 & -\mu \\
    0 & \cdots & 0 & m
  \end{pmatrix}
  \mathfrak{B}_1
\end{equation}
for unique $m > 0$ (since $\det \gamma = +1$) and $\mu \pmod m$.

It remains to show that if we have (\ref{eq:converseparameterization2}), then necessarily $\xi \in I_l^{-1}$.
To see this we make use of the notation established in the introduction.
We recall that $\{ \overline{\beta}_{l1}, \dots, \overline{\beta}_{ld} \}$ is a fixed basis of $I_l^{-1}$, the integers $b_{ijkl}$ are given by
\begin{equation}
  \label{eq:bijkldef1}
  \overline{\beta}_{li} \beta_{lj} = \sum_{1\leq k \leq d} b_{ijkl} \alpha^{d-k},
\end{equation}
and the matrices $B_{li}$ have entries $b_{ijkl}$, $j$ indexing the rows and $k$ indexing the columns.

We note that $\{\overline{\beta}_{l1}, \dots, \overline{\beta}_{ld} \}$ forms an $\mathbb{R}$-basis of $C_{r_1,r_2}$, and as such, there are $c_i \in \mathbb{R}$ such that
\begin{equation}
  \label{eq:cidef}
  \xi = \sum_{1\leq i \leq d} c_i \overline{\beta}_{li},
\end{equation}
where $\xi$ is as in (\ref{eq:converseparameterization}).
We now observe that $\xi \in I_l^{-1}$ if and only if all $c_i \in \mathbb{Z}$, $1\leq i \leq d$.

To verify this, we construct an integral matrix $B_l$ by taking the first column of $B_{li}$ as the $i$th column of $B_l$.
In view of the definition of the $B_{li}$, it is easily seen from (\ref{eq:converseparameterization3}) that we have
\begin{equation}
  \label{eq:Bliaction}
  \mathfrak{B}_l
  \begin{pmatrix}
    \xi^{(1)} & \cdots & 0 \\
    \vdots & \ddots & \vdots \\
    0 & \cdots & \xi^{(d)}
  \end{pmatrix}
  = \left( \sum_{1\leq i \leq d} c_i B_{li} \right)\mathfrak{B}_1,
\end{equation}
and hence
\begin{equation}
  \label{eq:Blaction}
  B_l
  \begin{pmatrix}
    c_1 \\
    c_2 \\
    \vdots \\
    c_d 
  \end{pmatrix}
  = \gamma^{-1}
  \begin{pmatrix}
    1 \\
    0 \\
    \vdots \\
    0
  \end{pmatrix}
  .
\end{equation}
Since $\gamma^{-1}$ is an integral matrix, the integrality of the $c_i$, and hence what remains to be proved of theorem \ref{theorem:parameterization}, is implied by the following lemma.
\begin{lemma}
  \label{lemma:B1}
  We have $B_l \in \mathrm{GL}(d,\mathbb{Z})$.
\end{lemma}

\begin{proof}
  We first show that $\det B_l \neq 0$.
  If it was the case that $\det B_l = 0$, then there would be integers $c_i'$, not all $0$, such that
  \begin{equation}
    \label{eq:nullspace}
    B_l
    \begin{pmatrix}
      c_1' \\
      \vdots \\
      c_d' 
    \end{pmatrix}
    =
    \begin{pmatrix}
      0 \\
      \vdots \\
      0 
    \end{pmatrix}
    .
  \end{equation}
  Setting
  \begin{equation}
    \label{eq:xiprimedef}
    \xi' = \sum_{1\leq i \leq d} c_i' \overline{\beta}_{li},
  \end{equation}
  we then have in view of (\ref{eq:Bliaction}) that
  \begin{equation}
    \label{eq:alphaspan}
    \xi' \beta_{li} \in \mathrm{Span} \{ 1, \alpha, \dots, \alpha^{d-2} \}
  \end{equation}
  for all $1\leq i \leq d$.
  Since $I_l^{-1}$ is a rank $d$ lattice in $C_{r_1,r_2}$ and $\mathbb{Q}(\alpha)$ has no zero-divisors, the $c_i'$ not being all $0$ implies that multiplication by $\xi'$ defines an invertible map on $C_{r_1,r_2}$, and so (\ref{eq:alphaspan}) contradicts the fact that the $\beta_{li}$ are linearly independent, i.e. that $I_l$ is a rank $d$ lattice in $C_{r_1,r_s}$.

  We now show that $\det B_l = \pm 1$.
  We first note that if $c_i'$ are integers, then, in view of (\ref{eq:Bliaction}),
  \begin{equation}
    \label{eq:idealbasis}
    \left( \sum_{1\leq i \leq d} c_i' B_{li} \right) \mathfrak{B}_1
  \end{equation}
  is a basis for an integral ideal $I$ in $\mathbb{Z}[\alpha]$.
  Let $B$ be the integral matrix resulting from putting $\sum c_i' B_{li}$ into Hermite normal form.
  Inspecting the first column, we observe that the first entry of $B$, $b_{11}$, can be expressed as
  \begin{equation}
    \label{eq:firstcolumngcd}
    b_{11} = \gcd B_l
    \begin{pmatrix}
      c_1' \\
      \vdots \\
      c_d'
    \end{pmatrix}
    .
  \end{equation}

  On the other hand, proposition \ref{proposition:invariantfactors} implies that $b_{11}$ divides the ideal $I = \xi' I_l^{-1}$, where $\xi' = \sum c_i' \overline{\beta}_{li}$.
  Since we are assuming that the fixed ideals $I_l^{-1}$ are not divisible by any rational integers, we see that $B_l$ maps primitive integral vectors into primitive integral vectors.
  By putting $B_l$ into Smith normal form, we observe that the diagonal matrix of the elementary divisors of $\mathbb{Z}^d / B_l \mathbb{Z}^d$ also has this property, and hence these elementary divisors are all equal to $1$.
  Thus, $\det B_l = \pm 1$ as required.
\end{proof}

We now turn to the proof of theorem \ref{theorem:parameterization1}.
In view of theorem \ref{theorem:parameterization}, we have that
\begin{equation}
  \label{eq:parameterization4}
  \begin{pmatrix}
    1 & \cdots & 0 & -\mu^{d-1} \\
    \vdots & \ddots & \vdots & \vdots \\
    0 & \cdots & 1 & -\mu \\
    0 & \cdots & 0 & m
  \end{pmatrix}
  \mathfrak{B}_1 = \gamma \mathfrak{B}_l
  \begin{pmatrix}
    \xi^{(1)} & \cdots & 0 \\
    \vdots & \ddots & \vdots \\
    0 & \cdots & \xi^{(d)}
  \end{pmatrix}
\end{equation}
connects $m$ and $\mu \pmod m$ with $\gamma \in \Gamma$ and $\xi \in C_{r_1,r_2}^+$.
Moreover, we have that $\xi \in I_l^{-1}$, and so
\begin{equation}
  \label{eq:xici}
  \xi = \sum_{1\leq i \leq d} c_i \overline{\beta}_{li},
\end{equation}
for some integers $c_i$
Conversely, given integers $c_i$, $\xi$ defined by (\ref{eq:xici}) is an element of $I_l^{-1}$, and we stipulate that this $\xi$ is in $\mathcal{D}$, a fixed fundamental domain for the action on $\mathbb{Q}(\alpha) \cap C_{r_1,r_2}^+$ of the totally positive units in $\mathbb{Z}[\alpha]$.

From the definition of the $B_{li}$ and the definition of $C$ in theorem \ref{theorem:parameterization1}, we have
\begin{equation}
  \label{eq:xiaction}
  \mathfrak{B}_{l}
  \begin{pmatrix}
    \xi^{(1)} & \cdots & 0 \\
    \vdots & \ddots & \vdots \\
    0 & \cdots & \xi^{(d)}
  \end{pmatrix}
  = C\mathfrak{B}_{1}.
\end{equation}
Inserting this into (\ref{eq:parameterization4}), we obtain
\begin{equation}
  \label{eq:gammaeq}
  \gamma
  \begin{pmatrix}
    1 & \cdots & 0 & -\mu^{d-1} \\
    \vdots & \ddots & \vdots & \vdots \\
    0 & \cdots & 1 & -\mu \\
    0 & \cdots & 0 & m
  \end{pmatrix}
  = C.
\end{equation}
Examining the first $d-1$ columns of this equation, we see that
\begin{equation}
  \label{eq:gammainverseform}
  \gamma^{-1} =
  \begin{pmatrix}
    c_{11} & \cdots & c_{1(d-1)} & * \\
    \vdots & \ddots & \vdots & \vdots \\
    c_{d1} & \cdots & c_{d(d-1)} & * 
  \end{pmatrix}
  ,
\end{equation}
where the $c_{ij}$ are the entries of $C$.

A $\gamma \in \mathrm{SL}(d,\mathbb{Z})$ satisfying (\ref{eq:gammainverseform}) exists if and only if the determinants of the $d$ minors obtained from the first $d-1$ columns are coprime, i.e. that
\begin{equation}
  \label{eq:coprimeminors}
  \gcd( \det C_{1d}, \dots, \det C_{dd}) = 1
\end{equation}
where $C_{ij}$ is the $(d-1)\times (d-1)$ matrix obtained from $C$ by removing the $i$th row and $j$th column.
We note that in view of proposition \ref{proposition:invariantfactors}, this coprimality condition provides a criterion for the ideal $I = \xi I_l$ to have cyclic quotient in terms of $\xi$.

Assuming (\ref{eq:coprimeminors}), there are integers $u_j$ such that
\begin{equation}
  \label{eq:ujdef}
  \sum_{1\leq j \leq d} (-1)^{j+d} u_j \det C_{jd} = 1,
\end{equation}
i.e.
\begin{equation}
  \label{eq:gammainverseform1}
  \gamma^{-1} =
  \begin{pmatrix}
    c_{11} & \cdots & c_{1(d-1)} & u_1 \\
    \vdots & \ddots & \vdots & \vdots \\
    c_{d1} & \cdots & c_{d(d-1)} & u_d
  \end{pmatrix}
  \in \mathrm{SL}(d,\mathbb{Z}).
\end{equation}
We note that such integers $u_j$ are defined by (\ref{eq:ujdef}) up to multiplying by matrices in $U \subset \mathrm{SL}(d,\mathbb{Z})$ on the right.
That is to say that if $u_j'$ are other integers satisfying (\ref{eq:ujdef}), then the corresponding $\gamma'$ satisfies
\begin{equation}
  \label{eq:gammaprimeeq}
  \gamma (\gamma')^{-1} =
  \begin{pmatrix}
    1 & \cdots & 0 & * \\
    \vdots & \ddots & \vdots & * \\
    0 & \cdots & 1 & * \\
    0 & \cdots & 0 & 1
  \end{pmatrix}
  .
\end{equation}
Returning to the equation (\ref{eq:gammaeq}), this ambiguity in $\gamma$ corresponds exactly to the ambiguity in picking representatives for the $\mu^j$ modulo $m$.
In other words, different choices for the set of integers $u_j$ satisfying (\ref{eq:ujdef}) correspond to different sets of representatives for the $\mu^j$ modulo $m$.

\subsection{Proof of theorems \ref{theorem:cubicparameterization} and \ref{theorem:cubicparameterization1}}
\label{sec:cubicparameterization1}

The proof of theorem \ref{theorem:cubicparameterization} is largely the same as the proof of theorem \ref{theorem:parameterization}.
Given roots $\mu_1 \pmod{m_1}$ and $\mu_2 \pmod {m_2}$ satisfying the conditions of theorem \ref{theorem:cubicparameterization}, there is a corresponding ideal $I$ by theorem \ref{theorem:cubiccorrespondence}.
By assumption, $I$ will be invertible in $\mathbb{Z}[\alpha]$, and so there is $\xi \in C_{r_1,r_2}^+$ such that $I = \xi I_l$ for some $l$.
This gives the natural basis
\begin{equation}
  \label{eq:generatorbasis}
  \begin{pmatrix}
    \xi \beta_{l1} \\
    \xi \beta_{l2} \\
    \xi \beta_{l3} 
  \end{pmatrix}
  = \mathfrak{B}_l
  \begin{pmatrix}
    \xi^{(1)} & 0 & 0 \\
    0 & \xi^{(2)} & 0 \\
    0 & 0 & \xi^{(3)}
  \end{pmatrix}
  .
\end{equation}
Having the bases for $I$, the basis (\ref{eq:generatorbasis}) and the basis given in theorem \ref{theorem:correspondence}, we conclude that there must be a $\gamma \in \mathrm{GL}(3,\mathbb{Z})$ such that
\begin{equation}
  \label{eq:cubicgammaeq}
  \begin{pmatrix}
    1 & \mu_1 + a_1 & * \\
    0 & m_1 & -\mu_2 m_2 \\
    0 & 0 & m_1m_2
  \end{pmatrix}
  \begin{pmatrix}
    (\alpha^{(1)})^2 &  (\alpha^{(2)})^2 & (\alpha^{(3)})^2 \\
    \alpha^{(1)} & \alpha^{(2)} & \alpha^{(3)} \\
    1 & 1 & 1
  \end{pmatrix}
  = \gamma \mathfrak{B}_l
  \begin{pmatrix}
    \xi^{(1)} & 0 & 0 \\
    0 & \xi^{(2)} & 0 \\
    0 & 0 & \xi^{(3)}
  \end{pmatrix}
  ,
\end{equation}
and in fact, since $m_2 > 0$, $\xi \in C_{r_1,r_1}^+$, and
\begin{equation}
  \label{eq:detsigncondition}
  \mathrm{sign} \det
  \begin{pmatrix}
    (\alpha^{(1)})^2 &  (\alpha^{(2)})^2 & (\alpha^{(3)})^2 \\
    \alpha^{(1)} & \alpha^{(2)} & \alpha^{(3)} \\
    1 & 1 & 1
  \end{pmatrix}
  = \mathrm{sign}\det \mathfrak{B}_l,
\end{equation}
we have $\gamma \in \Gamma = \mathrm{SL}(3,\mathbb{Z})$.

We note that multiplying (\ref{eq:cubicgammaeq}) on the left by matrices in $\Gamma_\infty$ does not change the residue classes $\mu_1 \pmod {m_1}$ and $\mu_2 \pmod {m_2}$, and conversely any representative of these classes can be obtained from a given one by multiplying by an appropriate element of $\Gamma_\infty$.
Moreover, an $\xi_1 \in C_{r_1,r_2}^+$ also satisfies $I = \xi_1 I_l$ if and only if $\xi_1 \xi^{-1} $ is a unit in $\mathbb{Z}[\alpha] \cap C_{r_1,r_2}^+$, and so
\begin{equation}
  \label{eq:unitchange}
  \mathfrak{B}_l
  \begin{pmatrix}
    \frac{\xi_1^{(1)}}{\xi^{(1)}} & 0 & 0 \\
    0 & \frac{\xi_1^{(2)}}{\xi^{(2)}} & 0 \\
    0 & 0 & \frac{\xi_1^{(3)}}{\xi^{(3)}} 
  \end{pmatrix}
  \mathfrak{B}_l^{-1} \in \Gamma.
\end{equation}
Hence by the definition of $\Gamma_l$, (\ref{eq:Gammaldef}), different choices of $\xi \in C_{r_1,r_2}^+$ satisfying $I = \xi I_l$ correspond in (\ref{eq:cubicgammaeq}) to different representatives of the coset $\gamma \Gamma_l$.

For the converse part of theorem \ref{theorem:cubicparameterization}, we observe that from the proof of the converse part of theorem \ref{theorem:parameterization1}, in particular lemma \ref{lemma:B1}, we have for $\gamma \in \Gamma$ and $\xi \in C_{r_1,r_2}^+$,
\begin{equation}
  \label{eq:cubicgammaeq1}
  \gamma^{-1}
  \begin{pmatrix}
    1 & * & * \\
    0 & * & * \\
    0 & * & * 
  \end{pmatrix}
  \mathfrak{B}_l = \mathfrak{B}_l
  \begin{pmatrix}
    \xi^{(1)} & 0 & 0 \\
    0 & \xi^{(2)} & 0 \\
    0 & 0 & \xi^{(3)}
  \end{pmatrix}
\end{equation}
implies that $\xi$ is a primitive vector in $I_l^{-1}$.
Accordingly, $I = \xi I_l$ is an ideal in $\mathbb{Z}[\alpha]$ not divisible by any rational integers.
Hence if $\gamma$ in fact satisfies (\ref{eq:conversecubicparameterization}), then theorem \ref{theorem:cubiccorrespondence} implies that  (\ref{eq:conversecubicparameterization1}) holds for some roots unique $\mu_1 \pmod {m_1}$, $\mu_2 \pmod {m_2}$.

We now proceed to the proof of theorem \ref{theorem:cubicparameterization1}.
From the above we have that if $\gamma$ satisfies
\begin{equation}
  \label{eq:cubicgammaeq2}
  \gamma \mathfrak{B}_l
  \begin{pmatrix}
    \xi^{(1)} & 0 & 0 \\
    0 & \xi^{(2)} & 0 \\
    0 & 0 & \xi^{(3)}
  \end{pmatrix}
  =
  \begin{pmatrix}
    1 & * & * \\
    0 & * & * \\
    0 & * & * 
  \end{pmatrix}
  \mathfrak{B}_1
\end{equation}
for some $\xi \in C_{r_1,r_2}^+$ then $\xi \in I_l^{-1}$, where we recall that
\begin{equation}
  \label{eq:B1recall}
  \mathfrak{B}_1 =
  \begin{pmatrix}
    (\alpha^{(1)})^2 & (\alpha^{(2)})^2 & (\alpha^{(3)})^2 \\
    \alpha^{(1)} & \alpha^{(2)} & \alpha^{(3)} \\
    1 & 1 & 1
  \end{pmatrix}
  .
\end{equation}
It follows that if $\gamma$ satisfies (\ref{eq:cubicgammaeq2}) then the first column of $\gamma^{-1}$ is the same as the first entries of $C$, which we recall is defined by
\begin{equation}
  \label{eq:Crecall}
  \mathfrak{B}_l
  \begin{pmatrix}
    \xi^{(1)} & 0 & 0 \\
    0 & \xi^{(2)} & 0 \\
    0 & 0 & \xi^{(3)} 
  \end{pmatrix}
  = C \mathfrak{B}_1
\end{equation}
or
\begin{equation}
  \label{eq:Crecall1}
  C = c_1 B_{l1} + c_2 B_{l2} + c_3 B_{l3}
\end{equation}
where $\xi = c_1 \overline{\beta}_{l1} + c_2 \overline{\beta}_{l2} + c_3 \overline{\beta}_{l3}$.

Now the equation
\begin{equation}
  \label{eq:cubicgammaeq3}
  \gamma C \mathfrak{B}_1 = \gamma \mathfrak{B}_l
  \begin{pmatrix}
    \xi^{(1)} & 0 & 0 \\
    0 & \xi^{(2)} & 0 \\
    0 & 0 & \xi^{(3)}
  \end{pmatrix}
  = 
  \begin{pmatrix}
    1 & * & * \\
    0 & * & * \\
    0 & 0 & * 
  \end{pmatrix}
  \mathfrak{B}_1
\end{equation}
puts additional constraints on the bottom row of $\gamma$.
We find that these constraints are best understood by considering $\xi' \in I_l$ such that $\xi'\xi \in \mathbb{Z}$.
If we define the integer matrix $C'$ by
\begin{equation}
  \label{eq:Cprimedef}
  C'\mathfrak{B}_l = \mathfrak{B}_1
  \begin{pmatrix}
    (\xi')^{(1)} & 0 & 0 \\
    0 & (\xi')^{(2)} & 0 \\
    0 & 0 & (\xi')^{(3)}
  \end{pmatrix}
  ,
\end{equation}
or $C' = c_1B_{l1}' + c_2 B_{l2}' + c_3B_{l3}'$, where $B_{li}'$ are defined via (\ref{eq:bijklprimedef}), then $\xi \xi' \in \mathbb{Z}$ if and only if
\begin{equation}
  \label{eq:CCprimeeq}
  C'C =
  \begin{pmatrix}
    * & * & * \\
    * & * & * \\
    0 & 0 & * 
  \end{pmatrix}
  .
\end{equation}
Hence if $\gamma$ satisfies (\ref{eq:cubicgammaeq3}), then the third row of $\gamma $ must be proportional to the third row of $C'$.

We can resolve ambiguities of sign by requiring that $\xi\xi'$ be a positive integer, so the more difficult condition to ensure is that the entries of the bottom row of $C'$ are coprime.
Our solution is just to divide $\xi'$ by the common divisor of these integers, and lemma \ref{lemma:Blprime} below shows that the result of this division is still be an element of $I_l$.
We construct the matrix $B_l'$ by taking the third row of $B_{li}'$ as the $i$th row of $B_l'$, so that
\begin{equation}
  \label{eq:Blprimedef}
  \begin{pmatrix}
    c_1' & c_2' & c_3'
  \end{pmatrix}
  B_l' =
  \begin{pmatrix}
    c_{31}' & c_{32}' & c_{33}'
  \end{pmatrix}
  .
\end{equation}
Then we have the following lemma, which implies that if the $c_{3j}'$ are integers, then so are the $c_j'$.

\begin{lemma}
  \label{lemma:Blprime}
  $B_l' \in \mathrm{GL}(3,\mathbb{Z})$.
\end{lemma}

\begin{proof}
  The proof of this lemma is similar to that of lemma \ref{lemma:B1}, and just as in that proof, $B_l'$ has integer entries and is easily seen to have nonzero determinant.
  Hence it suffices to show that $B_l'$ maps primitive vectors to primitive vectors.
  Accordingly, we set $\xi' = c_1' \beta_{l1} + c_2'\beta_{l2} + c_3'\beta_{l3}$ with $\gcd(c_1', c_2', c_3') = 1$, and we observe that this implies that the ideal $\xi' I_l^{-1} \subset \mathbb{Z}[\alpha]$ is not divisible by any rational integers.

  We have
  \begin{equation}
    \label{eq:Cprimeeq}
    C' \mathfrak{B}_l = \mathfrak{B}_1
    \begin{pmatrix}
      (\xi')^{(1)} & 0 & 0 \\
      0 & (\xi')^{(2)} & 0 \\
      0 & 0 & (\xi')^{(3)}
    \end{pmatrix}
    ,
  \end{equation}
  and so
  \begin{equation}
    \label{eq:transposeC'primeeq}
    \transpose{\mathfrak{B}_l}^{-1}
    \begin{pmatrix}
      (\xi')^{(1)} & 0 & 0 \\
      0 & (\xi')^{(2)} & 0 \\
      0 & 0 & (\xi')^{(3)}
    \end{pmatrix}
    = \transpose{C}' \transpose{\mathfrak{B}_1}^{-1}.
  \end{equation}
  It is well known, see for example \cite{ConradNote}, that since $\mathbb{Z}[\alpha]$ is monogenic, $\transpose{\mathfrak{B}_1}^{-1}$ is a basis for the fractional ideal generated by $\frac{1}{F'(\alpha)}$.
  In fact we have
  \begin{multline}
    \label{eq:different}
    \transpose{
      \begin{pmatrix}
        (\alpha^{(1)})^2 & (\alpha^{(2)})^2 & (\alpha^{(3)})^2 \\
        \alpha^{(1)} & \alpha^{(2)} & \alpha^{(3)} \\
        1 & 1 & 1
      \end{pmatrix}
      ^{-1}} \\
    =
    \begin{pmatrix}
      0 & 0 & 1 \\
      0 & 1 & a_1 \\
      1 & a_1 & a_2
    \end{pmatrix}
    \begin{pmatrix}
      (\alpha^{(1)})^2 & (\alpha^{(2)})^2 & (\alpha^{(3)})^2 \\
      \alpha^{(1)} & \alpha^{(2)} & \alpha^{(3)} \\
      1 & 1 & 1
    \end{pmatrix}
    \begin{pmatrix}
      F'(\alpha^{(1)})^{-1} & 0 & 0 \\
      0 & F'(\alpha^{(2)})^{-1} & 0 \\
      0 & 0 & F'(\alpha^{(3)})^{-1} 
    \end{pmatrix}
    .
  \end{multline}
  It follows that
  \begin{equation}
    \label{eq:Cprimeideal}
    \transpose{C'}
    \begin{pmatrix}
      0 & 0 & 1 \\
      0 & 1 & a_1 \\
      1 & a_1 & a_2
    \end{pmatrix}
    \mathfrak{B}_1
  \end{equation}
  is a basis for an integral ideal not divisible by any rational integers, and so, following the proof of lemma \ref{lemma:B1}, proposition \ref{proposition:invariantfactors} implies that the first column of
  \begin{equation}
    \label{eq:firstcolumn1}
     \transpose{C'}
    \begin{pmatrix}
      0 & 0 & 1 \\
      0 & 1 & a_1 \\
      1 & a_1 & a_2
    \end{pmatrix}
    ,
  \end{equation}
  i.e. the third row of $C'$, has coprime entries.  
\end{proof}

\section{Applications}
\label{sec:applications}

We now derive some consequences of the previous theorems, propositions \ref{proposition:generalapproximation}, \ref{proposition:spacing}, \ref{proposition:cotypezeta}, \ref{proposition:composition}, and \ref{proposition:cubiccomposition}.
We first discuss in section \ref{sec:approximationspacing} propositions \ref{proposition:generalapproximation} and \ref{proposition:spacing} on approximations to the vector $\left( \frac{\mu^{d-j}}{m}\right)$ and bounds for the number of these vectors contained in small balls.
Then in section \ref{sec:cotypezeta} we find an explicit Euler product for the co-type zeta function for the cubic order $\mathbb{Z}[2^{1/3}]$, proposition \ref{proposition:cotypezeta}.
Finally, in section \ref{sec:composition} we prove propositions \ref{proposition:composition} and \ref{proposition:cubiccomposition} on the operations on the roots of the congruence corresponding to ideal composition. 

\subsection{Approximation and bounds for $\left(\frac{\mu^{d-j}}{m}\right)$}
\label{sec:approximationspacing}

To start our proof of proposition \ref{proposition:generalapproximation}, we first select the fundamental domain $\mathcal{D}$ to use in the application of theorem \ref{theorem:parameterization1}.
The property we require of $\mathcal{D}$ is that all the embeddings of $\xi \in \mathcal{D}$ have the same size as $N(\xi)^{\frac{1}{d}}$.
From Dirichlet's unit theorem, in the logarithmic embedding of $\mathbb{Z}[\alpha]\cap C_{r_1,r_2}^+$ into $\mathbb{R}^d$, the totally positive units cut out a rank $d-1$ lattice on the plane orthogonal to the vector $(1,1,\dots,1)$.
A fundamental domain for the action of the totally positive units can be taken to be the region that projects parallel to $(1,1,\dots,1)$ onto a fundamental parallelopiped of this lattice.
Taking this fundamental parallelopiped to be within a bounded distance of the origin, it is clear that if $\xi$ is in such a fundamental domain, then $(\log |\xi^{(1)}|, \dots, \log |\xi^{(d)}|)$ is within a bounded distance of its orthogonal projection onto the span of $(1,\dots, 1)$, i.e. $(\frac{\log N(\xi)}{d}, \dots, \frac{\log N(\xi)}{d})$.
It follows that $\xi^{(j)} \asymp N(\xi)^{1/d}$ as desired.

We now recall the relation between $\xi$ and $C$ from theorems \ref{theorem:parameterization} and \ref{theorem:parameterization1}, i.e.
\begin{equation}
  \label{eq:xiCrelation}
  \mathfrak{B}_l
  \begin{pmatrix}
    \xi^{(1)} & \cdots & 0 \\
    \vdots & \ddots & \vdots \\
    0 & \cdots & \xi^{(d)}
  \end{pmatrix}
  = C \mathfrak{B}_1.
\end{equation}
We have
\begin{equation}
  \label{eq:Crewrite}
  C = \mathfrak{B}_l^{-1}
  \begin{pmatrix}
    \xi^{(1)} & \cdots & 0 \\
    \vdots & \ddots & \vdots \\
    0 & \cdots & \xi^{(d)}
  \end{pmatrix}
  \mathfrak{B}_1,
\end{equation}
and also, setting $C_{ij}$ to be the minor of $C$ with $i$th row and $j$th column removed,
\begin{equation}
  \label{eq:xiinverse}
  \begin{pmatrix}
    (\xi^{(1)})^{-1} & \cdots & (\xi^{(d)})^{-1} 
  \end{pmatrix}
  = \frac{1}{m}
  \begin{pmatrix}
    (-1)^{d+1} \det C_{1d} & \cdots & \det C_{dd} 
  \end{pmatrix}
  \mathfrak{B}_l
\end{equation}
by examining the last row of
\begin{equation}
  \label{eq:lastrow}
  C^{-1} \mathfrak{B}_l = \mathfrak{B}_1
  \begin{pmatrix}
    \xi^{(1)} & \cdots & 0 \\
    \vdots & \ddots & \vdots \\
    0 & \cdots & \xi^{(d)}
  \end{pmatrix}
  ^{-1}
\end{equation}
and recalling that
\begin{equation}
  \label{eq:B1recall1}
  \mathfrak{B}_1 =
  \begin{pmatrix}
    (\alpha^{(1)})^{d-1} & \cdots & (\alpha^{(d)})^{d-1} \\
    \vdots & \ddots & \vdots \\
    \alpha^{(1)} & \cdots & \alpha^{(d)} \\
    1 & \cdots & 1
  \end{pmatrix}
  .
\end{equation}
From these equations, it is evident that the entries $c_{ij}$ of $C$ are fixed linear combinations of the embeddings of $\xi$ and that the embeddings of $\xi$ are fixed combinations of the $\frac{1}{m} \det C_{kd}$.
It follows that if the embeddings of $\xi$ are all $\asymp m^{1/d}$, then $c_{ij} \ll m^{1/d}$ for all $i$ and $j$, and $\det C_{kd} \gg m^{1 - 1/d}$ for some $k$.

The $k$ for which $\det C_{kd} \gg m^{1 - 1/d}$ holds is exactly the $k$ giving the approximation in proposition \ref{proposition:generalapproximation}.
However for the sake of exposition, we assume in what follows that $\det C_{dd} \gg m^{1 - 1/d}$, that is $k = d$, and leave the necessary modifications to the argument in the case $k < d$ to the reader.

Using the notation established in theorem \ref{theorem:parameterization1}, we recall that
\begin{equation}
  \label{eq:gammaCeq}
  \begin{pmatrix}
    1 & \cdots & 0 & -\mu^{d-1} \\
    \vdots & \ddots & \vdots & \vdots \\
    0 & \cdots & 1 & -\mu \\
    0 & \cdots & 0 & m 
  \end{pmatrix}
  = \gamma C,
\end{equation}
and we also recall that
\begin{equation}
  \label{eq:gammainverseform2}
  \gamma^{-1} =
  \begin{pmatrix}
    c_{11} & \cdots & c_{1(d-1)} & u_1 \\
    \vdots & \ddots & \vdots & \vdots \\
    c_{d1} & \cdots & c_{d(d-1)} & u_d
  \end{pmatrix}
  .
\end{equation}
Rearranging the last column of this matrix equation, we obtain
\begin{equation}
  \label{eq:lastcolumn}
  \begin{pmatrix}
    c_{11} & \cdots & c_{(d-1)(d-1)} \\
    \vdots & \ddots & \vdots \\
    c_{(d-1)1} & \cdots & c_{(d-1)(d-1)}
  \end{pmatrix}
  \begin{pmatrix}
    \frac{\mu^{d-1}}{m} \\
    \vdots \\
    \frac{\mu}{m}
  \end{pmatrix}
  = 
  \begin{pmatrix}
    u_1 \\ \vdots \\ u_{d-1} 
  \end{pmatrix}
  - \frac{1}{m}
  \begin{pmatrix}
    c_{1d} \\ \vdots \\ c_{(d-1)d}
  \end{pmatrix}
  ,
\end{equation}
in addition to $d-1$ other equations resulting from removing other rows besides the $d$th.
In the case $k < d$ one would consider the equation resulting from removing the $k$th row.

As $c_{ij} \ll m^{1/d}$, we can interpret these equations as the vector $\left( \frac{\mu^{d-j}}{m} \right) \in \mathbb{R}^{d-1} / \mathbb{Z}^{d-1}$ being close to the $d-1$  planes $c_{j1}X_1 + \cdots + c_{j(d-1)}X_{d-1} = u_j$, $1\leq j \leq d-1$.
Moreover, under the assumption that
\begin{equation}
  \label{eq:detcondition}
  \det C_{dd} \gg m^{1 - \frac{1}{d}},
\end{equation}
the vector in fact lies close to the intersection of these $d-1$ planes, thus verifying proposition \ref{proposition:generalapproximation} in this case.

From (\ref{eq:lastcolumn}) we have
\begin{equation}
  \label{eq:lastcolumn1}
  \frac{1}{m}
  \begin{pmatrix}
    \mu^{d-1} \\
    \vdots \\
    \mu
  \end{pmatrix}
  = C_{dd}^{-1}
  \begin{pmatrix}
    u_1 \\
    \vdots \\
    u_{d-1}
  \end{pmatrix}
  - \frac{1}{m} C_{dd}^{-1}
  \begin{pmatrix}
    c_{1d} \\
    \vdots \\
    c_{(d-1)d}
  \end{pmatrix}
  ,
\end{equation}
and so all that needs to be proved is that under the assumption \ref{eq:detcondition} we have
\begin{equation}
  \label{eq:errorterm}
  C_{dd}^{-1}
  \begin{pmatrix}
    c_{1d} \\ \vdots \\ c_{(d-1)d} 
  \end{pmatrix}
  \ll 1. 
\end{equation}
To verify (\ref{eq:errorterm}), we rearrange the last column of the equation $CC^{-1} = I$ to obtain
\begin{equation}
  \label{eq:errorterm2}
  C_{dd}
  \begin{pmatrix}
    (-1)^{d} \det C_{d1} \\
    \vdots \\
    \det C_{d(d-1)}
  \end{pmatrix}
  = \det C_{dd}
  \begin{pmatrix}
    c_{1d} \\ \vdots \\ c_{(d-1)d}
  \end{pmatrix}
  ,
\end{equation}
and so
\begin{equation}
  \label{eq:errorterm3}
  C_{dd}^{-1}
  \begin{pmatrix}
    c_{1d} \\
    \vdots \\
    c_{(d-1)d} 
  \end{pmatrix}
  = \frac{1}{\det C_{dd}}
  \begin{pmatrix}
    (-1)^d \det C_{d1} \\
    \vdots \\
    \det C_{d(d-1)}
  \end{pmatrix}
  .
\end{equation}
It is clear now that (\ref{eq:errorterm}) follows from $\det C_{dd} \gg m^{1 - 1/d}$ and $c_{ij} \ll m^{1/d}$.
We remark that in the analogous argument for $k < d$ being the one that satisfies $\det C_{kd} \gg m^{1 - 1/d}$, one rearranges the $k$th column of $CC^{-1} = I$.

\bigskip 

We now turn to the proof of proposition \ref{proposition:spacing}.
As discussed in the introduction, we begin with a discussion of the spacing properties between general rational points in $\mathbb{R}^{d-1}$ and later specialize to the approximations of $\left(\frac{\mu^{d-j}}{m}\right)$ given in proposition \ref{proposition:generalapproximation}.
Every rational point can be written uniquely in the form $\left( \frac{r_1}{q}, \dots, \frac{r_{d-1}}{q}\right)$ where $q$ is a positive integer and the $r_j$ and $q$ are coprime integers, i.e. $\gcd(r_1, \dots, r_{d-1}, q) = 1$.
We remark that written this way, $q$ is the torsion of the coset of the rational point in $\mathbb{R}^{d-1} / \mathbb{Z}^{d-1}$, and we refer to such a point as a $q$-torsion point.
We also note that the point $\left( \frac{r_1}{q}, \dots, \frac{r_{d-1}}{q}\right)$ is naturally identified with the point in projective space having homogeneous coordinates $(r_1, \dots, r_{d-1}, q)$.

Given two torsion points $\bm{r}$ and $\bm{r}'$, we consider the Pl\"ucker coordinates of the line containing both.
These coordinates are the quantities $s_{ij}$, $1\leq i < j \leq d$, formed by taking $2\times 2$ determinants from the matrix
\begin{equation}
  \label{eq:plucker}
  \begin{pmatrix}
    r_1 & \cdots & r_{d-1} & q \\
    r_1' & \cdots & r_{d-1}' & q'
  \end{pmatrix}
  ,
\end{equation}
i.e.
\begin{equation}
  \label{eq:sijdef}
  s_{ij} = \det
  \begin{pmatrix}
    r_i & r_j \\
    r_i' & r_j' 
  \end{pmatrix}
  \mathrm{\ if\ } i,j< d,\quad s_{id} = \det
  \begin{pmatrix}
    r_i & q \\
    r_i' & q'
  \end{pmatrix}
  .
\end{equation}
We remark that one should only consider these coordinates up to scalar multiplication in order for the line to determine the coordinates.
However, we find the distinction between multiples of Pl\"ucker coordinates to be useful, as indicated in the following observation:
\begin{equation}
  \label{eq:distance}
  \vert\vert \bm{r} - \bm{r}'\vert\vert = \frac{1}{qq'} \left(s_{1d}^2 + \cdots + s_{(d-1)d}^2 \right)^{1/2}. 
\end{equation}

Fixing a torsion point $\bm{r}$, we can lower bound the distance between $\bm{r}$ and any other torsion point by considering the set of all $(s_{id}) \in \mathbb{Z}^{d-1}$ formed as in (\ref{eq:sijdef}) as $\bm{r}'$ ranges over all torsion points.
We observe that since
\begin{equation}
  \label{eq:pluckeradd}
  \det
  \begin{pmatrix}
    r_i & q \\
    r_i' & q'
  \end{pmatrix}
  + \det
  \begin{pmatrix}
    r_i & q \\
    r_i'' & q''
  \end{pmatrix}
  = \det
  \begin{pmatrix}
    r_i & q \\
    r_i' + r_i'' & q' + q''
  \end{pmatrix}
  ,
\end{equation}
this set is additive and so forms a sublattice $\Lambda(\bm{r})$ of $\mathbb{Z}^{d-1}$ which is easily seen to have full rank.
Moreover, since
\begin{equation}
  \label{eq:pluckertranslate}
  \det
  \begin{pmatrix}
    r_j + kq & q \\
    r_j' & q' 
  \end{pmatrix}
  = \det
  \begin{pmatrix}
    r_j & q \\
    r_j' - kq' & q'
  \end{pmatrix}
  ,
\end{equation}
$\Lambda(\bm{r})$ only depends on the coset of $\bm{r}$ in $\mathbb{R}^{d-1} / \mathbb{Z}^{d-1}$.
Geometrically we think of $\Lambda(\bm{r})$ as being identified with the integral lines containing $\bm{r}$, and one can work around the caveats mentioned in the previous paragraph by working only with the primitive elements in $\Lambda(\bm{r})$.
We record these observations together with (\ref{eq:distance}) in the following lemma.

\begin{lemma}
  \label{lemma:torsionpointspacing}
  Let $Q$ be a positive real number and let $\bm{r}$ be a $q$-torsion point in $\mathbb{R}^{d-1}/\mathbb{Z}^{d-1}$.
  Then the distance between $\bm{r}$ and any distinct torsion point with torsion $\leq Q$ is at least
  \begin{equation}
    \label{eq:torsionpointspacing}
    \frac{1}{qQ} \min \{ ||\bm{v}||\ :\ \bm{v}\in\Lambda(\bm{r}), \bm{v}\neq 0\}. 
  \end{equation}
\end{lemma}

For our purposes of proving proposition \ref{proposition:spacing}, we take the point on the right side of (\ref{eq:approximation}) as $\bm{r}$.
As above, we assume that the $k$ giving the approximation to $\left( \frac{\mu^{d-j}}{m} \right)$ is $k=d$, namely
\begin{equation}
  \label{eq:torsionpoint}
  \bm{r} = C_{dd}^{-1}
  \begin{pmatrix}
    u_1 \\
    \vdots \\
    u_{d-1}
  \end{pmatrix}
  ,
\end{equation}
which has torsion $|\det C_{dd}|$.
This expression naturally gives $\bm{r}$ as the intersection of the $d-1$ planes
\begin{equation}
  \label{eq:planes}
  c_{i1}X_1 + \cdots + c_{i(d-1)}X_{d-1} = u_i,
\end{equation}
$1\leq i \leq d -1 $, and so it is convenient to consider the integral lines containing $\bm{r}$ dually as the intersection of sets of $d-2$ hyper-planes containing $\bm{r}$.
In the case $k < d$< we would view the point as the intersection of the planes (\ref{eq:planes}) for $i \neq k$. 

The lattice $\Lambda(\bm{r})$ can be determined from this dual perspective as well.
If $\bm{r}'$ is another torsion point contained in the first $d-2$ hyperplanes, then we have
\begin{equation}
  \label{eq:dualplucker1}
  \begin{pmatrix}    
    c_{11} & \cdots & c_{1(d-1)} & u_1 \\    
    \vdots & \ddots & \vdots & \vdots \\    
    c_{(d-2)1} & \cdots & c_{(d-2)(d-1)} & u_{d-2}    
  \end{pmatrix}
  \begin{pmatrix}
    r_1 & r_1' \\
    \vdots & \vdots \\
    r_{d-1} & r_{d-1}' \\
    q & q'
  \end{pmatrix}
  =
  \begin{pmatrix}
    0 & 0 \\
    \vdots & \vdots \\
    0 & 0 
  \end{pmatrix}
  .
\end{equation}
We let $\gamma \in \mathrm{SL}(d, \mathbb{Z})$ be the matrix from theorems \ref{theorem:parameterization} and \ref{theorem:parameterization1}, so that the matrix on the left of (\ref{eq:dualplucker1}) is the first $d-2$ rows of $\gamma$.
Then (\ref{eq:dualplucker1}) implies that
\begin{equation}
  \label{eq:gammarrprime}
  \gamma
  \begin{pmatrix}
    r_1 & r_1' \\
    \vdots & \vdots \\
    r_{d-1} & r_{d-1}' \\
    q & q' 
  \end{pmatrix}
  =
  \begin{pmatrix}
    0 & 0 \\
    \vdots & \vdots \\
    0 & 0 \\
    * & * \\
    * & * 
  \end{pmatrix}
  .
\end{equation}
If the Pl\"ucker coordinates $s_{ij}$ corresponding to $\bm{r}'$ give a primitive vector in $\Lambda(\bm{r})$, then $\gcd( s_{ij}) = 1$, and as this condition is preserved by the action of $\mathrm{SL}(d, \mathbb{Z})$, we have that the $2\times 2$ determinant in the bottom of the matrix on the right of (\ref{eq:gammarrprime}) is equal to $\pm 1$.
It follows that there is a $\gamma ' \in \mathrm{GL}(d, \mathbb{Z})$ having the same first $d-2$ rows as $\gamma$ but also having $\bm{r}$, $\bm{r}'$ as the last two columns of $\gamma^{-1}$.
It now follows from the Jacobi's equality between complementary co-factors of a matrix and its inverse that the Pl\"ucker coordinates $s_{ij}$ are equal (up to sign) to the determinants of the $(d-2)\times (d-2)$ minors of the matrix
\begin{equation}
  \label{eq:dualplucker2}
  \begin{pmatrix}
    c_{11} & \cdots & c_{1(d-1)} & u_1 \\
    \vdots & \ddots & \vdots & \vdots \\
    c_{(d-2)1} & \cdots & c_{(d-2)(d-1)} & u_{d-2} 
  \end{pmatrix}
\end{equation}
obtained by removing two columns. 

These observations, repeated for all choices of $d-2$ hyperplanes from the $d-1$, shows that $\Lambda$ is generated by the columns of the matrix $(\det C_{dd}) C_{dd}^{-1}$, which we denote by $\bm{c}_{j}$. 
Having this expression for a basis of $\Lambda(\bm{r})$ allows us to control the size of its shortest, nonzero vector, at least under the same hypothesis under which $\bm{r}$ is guaranteed to approximate $\left( \frac{\mu^{d-j}}{m} \right)$, i.e. if $i = d$ in the context of proposition \ref{proposition:generalapproximation}.

\begin{lemma}
  \label{lemma:latticecompact}
  If $\det C_{dd} \gg m^{1-1/d}$, then the matrix $(\det C_{dd}) C_{dd}^{-1}$, normalized by $(\det C_{dd})^{- (d-2)/(d-1)}$ to have determinant $1$, lies in a fixed, compact subset of $\mathrm{SL}(d-1,\mathbb{R})$, and hence the smallest vector in $\Lambda$ has size $\gg (\det C_{dd})^{(d-2)/(d-1)} \gg m^{1-2/d}$. 
\end{lemma}

\begin{proof}
  It is easy to see that the determinant of the matrix $(\det C_{dd}) C_{dd}^{-1}$ is $(\det C_{dd})^{d-2}$, so Hadamard's inequality implies
  \begin{equation}
    \label{eq:hadamard}
    (\det C_{dd}) ^{d-2} \leq ||\bm{c}_{1}|| \cdots || \bm{c}_{(d-1)}||.
  \end{equation}
  On the other hand, since each of the $\bm{c}_{j}$ have coordinates polynomials of degree $d-2$ in the $c_i$, which we recall are $\ll m^{1/d}$, we have $||\bm{c}_{j}|| \ll m^{(d - 2)/d}$.
  Replacing all but one of the $||\bm{c}_{j}||$ in (\ref{eq:hadamard}) by this bound, we have that, under the hypothesis $\det C_{dd} \gg m^{1-1/d}$,
  \begin{equation}
    \label{eq:basisbounds}
    m^{(d-1)(d-2)/d} \ll ||\bm{c}_{j}|| m^{(d-2)^2/d} \ll m^{(d-1)(d-2)/d},
  \end{equation}
  whence $||\bm{c}_{j}|| \asymp m^{1-2/d} \asymp (\det C_{dd})^{(d-2)/(d-1)}$.
  These estimates are enough to show that upon normalizing so that the determinant is $1$, the resulting matrix is in a compact subset of $\mathrm{SL}(d-1, \mathbb{R})$.

  Clearly then the normalized lattice obtained from $\Lambda$, having this basis, lies in a compact subset of $\mathrm{SL}(d-1, \mathbb{Z}) \backslash \mathrm{SL}(d-1, \mathbb{R})$, and so does not approach the cusp.
  This shows that the smallest vector in $\Lambda$ has size $\gg$ the $(d-1)$th root of the determinant, finishing the proof of the lemma.
\end{proof}

In view of the above lemma on the spacing between torsion points, lemma \ref{lemma:torsionpointspacing}, the proof of proposition \ref{proposition:spacing} is almost finished.
Indeed, for each of the points $\left( \frac{\mu^{d-j}}{m} \right)$ contained in a ball of radius $\frac{1}{M}$, all the approximations given by proposition \ref{proposition:generalapproximation} are contained in a ball of radius $O(\frac{1}{M})$.
However, by lemmas \ref{lemma:latticecompact} and \ref{lemma:torsionpointspacing}, each of these approximations are spaced from each other by at least
\begin{equation}
  \label{eq:spacing}
  \gg \frac{1}{M^{2(d-1)/d}} M^{(d-2)/d} = \frac{1}{M}.
\end{equation}
Hence there can be at most $\ll 1$ of these approximations in this ball.
The theorem is then proved if we can show that at most $\ll 1$ of the points $\left( \frac{\mu^{d-j}}{m} \right)$ can correspond to a given one of the approximations.

We start by noting that a torsion point in $\mathbb{R}^{d-1}/ \mathbb{Z}^{d-1}$ determines the corresponding lattice $\Lambda(\bm{r})$, and because $\mathrm{SL}(d-1,\mathbb{Z})$ acts discontinuously on $\mathrm{SL}(d-1,\mathbb{R})$, the number of bases of $\Lambda$ lying in the compact set of lemma \ref{lemma:latticecompact} is bounded by a constant that depends only on the compact set.
For each one of these bases, there are $d$ cases to consider, one for each of the possible $i$ giving the approximation of proposition \ref{proposition:generalapproximation}.
Further, we need to consider each of the narrow ideal classes, but once these possibilities are accounted for, which for our purposes just multiplies the final bound by a constant, we claim that the basis of the lattice determines the $c_i$ and whence the $m$ and $\mu\pmod m$.
This would indeed show that the number is bounded by a constant depending only on the congruence.

To see this final step, we recall that the matrix $C$ is a linear combination of matrices $B_{li}$ depending on the ideal class, and the coefficients are exactly the $c_i$.
So, continuing to work only with the case $i=d$ from proposition \ref{proposition:generalapproximation}, the question of recovering the $c_i$ from the matrix $C_{dd}$ is a question about the linear independence of the corresponding $(d-1)\times (d-1)$ sub-matrices of the $B_{li}$ obtained by removing the $d$th row and column.

Denoting these sub-matrices by $B_{1id}$, we suppose there are numbers $c_i$ so that
\begin{equation}
  \label{eq:lineardependence}
  \sum_{i=1}^d c_i B_{1id} =
  \begin{pmatrix}
    0 & \cdots & 0 \\
    \vdots & \ddots & \vdots \\
    0 & \cdots & 0
  \end{pmatrix}
  .
\end{equation}
From the definition of the $B_{li}$ this means that the corresponding $\xi = \sum c_i \overline{\beta}_i$ satisfies
\begin{equation}
  \label{eq:lineardependence2}
  \xi \beta_j \in \mathbb{Z}, \quad 1\leq j \leq d-1. 
\end{equation}
However, for a fixed $\xi$, the set of points $\beta \in C_{r_1,r_2}$ with $\xi \beta \in \mathbb{R}$ forms a line in $C_{r_1,r_2}$.
On the other hand, since the $\beta_j$ span a full rank lattice, namely $I_l$, at most one of the $\beta_j$ can lie on such a line.
This is a clearly contradicts (\ref{eq:lineardependence2}) when $d>2$, and so it remains to examine the case $d=2$.

This case is the most interesting because the $B_{112}, B_{122}$ are not linearly independent, they are just two numbers, and so one cannot recover $c_1, c_2$ from a linear combination $c_{11} = c_1B_{112} + c_2B_{122}$.
Instead we make use of the additional information contained in the numerator of the approximation $\frac{u_1}{c_{11}}$.
Since $\gamma^{-1} = 
\begin{pmatrix}
  c_{11} & u_1 \\
  c_{21} & u_2
\end{pmatrix}
\in \mathrm{SL}(2, \mathbb{Z})$, we have $c_{21} \equiv \overline{u}_1 \pmod {c_{11}}$.
Moreover $c_{11} \asymp m^{\frac{1}{2}}$ and $c_{21} \ll m^{\frac{1}{2}}$, and so we can recover $c_{21}$ from $\frac{u_1}{c_{11}}$ up to $O(1)$ possibilities.
Finally, as
\begin{equation}
  \label{eq:c1c2}
  \begin{pmatrix}
    c_{11} \\
    c_{21}
  \end{pmatrix}
  = B_l
  \begin{pmatrix}
    c_1 \\
    c_2
  \end{pmatrix}
\end{equation}
and $B_l \in \mathrm{GL}(2, \mathbb{Z})$, we recover $c_1, c_2$ up to $O(1)$ possibilities.

\subsection{Co-type zeta function for cubic orders}
\label{sec:cotypezeta}

For an ideal $I \subset \mathbb{Z}[2^{1/3}]$, we let $N_1(I)$, $N_2(I)$, and $N_3(I)$ denote the invariant factors.
That is
\begin{equation}
  \label{eq:invariantfactorsdef}
  \mathbb{Z}[2^{1/3}] / I \cong \mathbb{Z}/N_1(I) \mathbb{Z} \oplus \mathbb{Z}/N_2(I)\mathbb{Z} \oplus \mathbb{Z}/N_3(I)\mathbb{Z}
\end{equation}
with $N_3(I) \mid N_2(I)\mid N_2(I)$.
We define the co-type zeta function as
\begin{equation}
  \label{eq:cotypedef}
  \zeta_{\mathbb{Z}[2^{1/3}]} (s_1, s_2, s_3) = \sum_{0\neq I \subset \mathbb{Z}[2^{1/3}]} N_1(I)^{-s_1} N_2(I)^{-s_2} N_3(I)^{-s_3},
\end{equation}
where the sum is over ideals $I$.

In the language of theorem \ref{theorem:cubiccorrespondence} and proposition \ref{proposition:invariantfactors}, we see that $N_3(I)$ is the largest integer divisor of $I$, and, applying the theorem to $I/N_3(I)$, we have $m_2 = N_1(I)/N_2(I)$, $m_1 = N_2(I)/N_3(I)$.
We have
\begin{equation}
  \label{eq:cotype2}
  \begin{split}
    \zeta_{\mathbb{Z}[2^{1/3}]}(s_1,s_2,s_3) & = \sum_{0\neq I \subset \mathbb{Z}[2^{1/3}]} m_2(I)^{-s_1} m_1(I)^{-s_1 - s_2} N_3(I)^{-s_1 -s_2 -s_3} \\
    & = \zeta(s_1 + s_2 + s_3) \sum_{\substack{0\neq I \subset \mathbb{Z}[2^{1/3}] \\ l\nmid I, \forall l\in \mathbb{Z}}} m_2(I)^{-s_1} m_1(I)^{-s_1-s_2}, \\
  \end{split}
\end{equation}
where $\zeta(s)$ is the Riemann zeta function. We apply theorem \ref{theorem:cubiccorrespondence} to arrange this sum as
\begin{equation}
  \label{eq:cotype3}
  \begin{split}
    \zeta_{\mathbb{Z}[2^{1/3}]}(s_1, s_2, s_3) & = \sum_{\gcd(m_2,6) = 1}\sum_{\mu_2^3\equiv 2(m_2)} m_2^{-s_1} \sum_{m_1\geq 1}\sum_{\substack{\mu_1^3 \equiv 2 (m_1) \\ \gcd(m_1,m_2,\mu_1-\mu_2)=1}} m_1^{-s_1 - s_2} \\
    & \quad + \sum_{\gcd(m_2,6)=2} \sum_{\mu_2^3\equiv 2(m_2)} m_2^{-s_1} \sum_{\gcd(m_1,2)=1} \sum_{\substack{\mu_1^3\equiv 2(m_1) \\ \gcd(m_1,m_2,\mu_1-\mu_2)=1}} m_1^{-s_1 -s_2} \\
    & \quad + \sum_{\gcd(m_2,6)=3} \sum_{\mu_2^3\equiv 2(m_2)} m_2^{-s_1}\sum_{\gcd(m_1,3)=1} \sum_{\substack{\mu_1^3\equiv 2(m_1) \\ \gcd(m_1,m_2,\mu_1-\mu_2)=1}} m_1^{-s_1 - s_2} \\
    & \quad + \sum_{\gcd(m_2,6)=6} \sum_{\mu_2^3\equiv 2(m_2)} m_2^{-s_1} \sum_{\gcd(m_1,6)=1} \sum_{\substack{\mu_1^3 \equiv 2 (m_1) \\ \gcd(m_1,m_2,\mu_1-\mu_2)=1}} m_1^{-s_1 -s_2} \\
    & = S_1 + S_2 + S_3 + S_4,
  \end{split}
\end{equation}
say.

Starting with $S_1$, we first note that by the Chinese remainder theorem, the Dirichlet series in $m_1$ (with variable $s_1 + s_2$) has multiplicative coefficients, so we may consider each prime separately.
Since $\gcd(m_2, 6) = 1$, we find that Euler factors at $2$ and $3$ are $1 + 2^{-s_1 - s_2}$ and $1 + 3^{-s_1 - s_2}$.
For primes $p$ having $3$ distinct roots, the Euler factor at $p$ is either
\begin{equation}
  \label{eq:P1eulerfactor}
  1 + 2p^{-s_1 - s_2} + 2p^{-2s_1 - 2s_2} + \cdots = 1 + \frac{2 p^{-s_1 - s_2}}{1 - p^{-s_1 - s_2}}
\end{equation}
or
\begin{equation}
  \label{eq:P1eulerfactor1}
  1 + 3p^{-s_1 - s_2} + 3 p^{-2s_1 - 2s_2} + \cdots = 1 + \frac{3 p^{-s_1 - s_2}}{1 - p^{-s_1 - s_2}}
\end{equation}
depending on whether $p$ divides $m_2$ or not.
For primes $p \neq 2$ or $3$, we find that the Euler factor is just $1$ if $p\mid m_2$ and it is $(1 - p^{-s_1 - s_2})^{-1}$ if $p \nmid m_2$.
As the Euler factor for primes having no roots is just $1$, we conclude that
\begin{equation}
  \label{eq:S1}
  \begin{split}
    S_1  = & \sum_{\gcd(m_2,6)=1}\sum_{\mu_2^3\equiv 2(m_2)} m_2^{-s_1} (1 + 2^{-s_1 -s_2})(1 + 3^{-s_1 -s_2}) \\
    & \quad \times \prod_{\substack{p\mid m_2 \\ p\in\mathcal{P}_1}} \left( 1 + \frac{2p^{-s_1 -s_2}}{1 - p^{-s_1-s_2}} \right) \prod_{\substack{p\nmid m_2 \\ p \in \mathcal{P}_1}} \left( 1 + \frac{3p^{-s_1 -s_2}}{1 - p^{-s_1 -s_2}} \right) \prod_{\substack{p\nmid m_2 \\ p\in \mathcal{P}_2}} \left(\frac{1}{1 - p^{-s_1 -s_2}}\right),
  \end{split}
\end{equation}
where $\mathcal{P}_1$ is the set of primes in $\mathbb{Z}$ that split completely in $\mathbb{Z}[2^{1/3}]$ and $\mathcal{P}_2$ is the set of those that factor into a degree $1$ times a degree $2$ prime; neither $\mathcal{P}_1$ nor $\mathcal{P}_2$ contain $2$ or $3$. Explicitly, we have $\mathcal{P}_2$ is the set of all primes other than $2$ that are $\equiv 2 \pmod 3$, and $\mathcal{P}_1$ is the set of all primes that can be represented by the binary quadratic form $X^2 + 27Y^2$. 

We arrange this as
\begin{equation}
  \label{eq:S12}
  \begin{split}
    S_1 & = (1 + 2^{-s_1 -s_2})(1 + 3^{-s_1 -s_2}) \prod_{p\in\mathcal{P}_1} \left( 1 + \frac{3p^{-s_1 -s_2}}{1-p^{-s_1-s_2}} \right) \prod_{p\in\mathcal{P}_2} \left( \frac{1}{1-p^{-s_1-s_2}}\right) \\
    & \quad \times \sum_{\gcd(m,6)=1} \sum_{\mu^3\equiv 2(m)} m^{-s_1} \prod_{\substack{p\mid m \\ p\in\mathcal{P}_1}} \left( \frac{1 + p^{-s_1 - s_2}}{1 + 2p^{-s_1-s_2}} \right) \prod_{\substack{p\mid m \\ p\in \mathcal{P}_2}} \left( 1 - p^{-s_1-s_2}\right) \\
    & = (1 + 2^{-s_1-s_2})(1 + 3^{-s_1 -s_2}) \\
    & \quad \times \prod_{p\in\mathcal{P}_1} \left( \frac{ 1 + 2p^{-s_1} + 2p^{-s_1 -s_2} + p^{-2s_1 -s_2}}{(1 - p^{-s_1})(1 - p^{-s_1 - s_2})} \right)\prod_{p\in\mathcal{P}_2} \left( \frac{1 - p^{-2s_1 - s_2}}{(1 - p^{-s_1})(1-p^{-s_1-s_2})}\right). 
  \end{split}
\end{equation}
Similar calculations show that
\begin{equation}
  \label{eq:S234}
  \begin{split}
    S_2 & = 2^{-s_1}(1 + 3^{-s_1 - s_2}) \\
    & \quad \times \prod_{p\in\mathcal{P}_1} \left( \frac{ 1 + 2p^{-s_1} + 2p^{-s_1 -s_2} + p^{-2s_1 -s_2}}{(1 - p^{-s_1})(1 - p^{-s_1 - s_2})} \right)\prod_{p\in\mathcal{P}_2} \left( \frac{1 - p^{-2s_1 - s_2}}{(1 - p^{-s_1})(1-p^{-s_1-s_2})}\right) \\
    S_3 & = 3^{-s_1}(1+ 2^{-s_1 - s_2}) \\
    & \quad \times \prod_{p\in\mathcal{P}_1} \left( \frac{ 1 + 2p^{-s_1} + 2p^{-s_1 -s_2} + p^{-2s_1 -s_2}}{(1 - p^{-s_1})(1 - p^{-s_1 - s_2})} \right)\prod_{p\in\mathcal{P}_2} \left( \frac{1 - p^{-2s_1 - s_2}}{(1 - p^{-s_1})(1-p^{-s_1-s_2})}\right) \\
    S_4 & = 6^{-s_1} \prod_{p\in\mathcal{P}_1} \left( \frac{ 1 + 2p^{-s_1} + 2p^{-s_1 -s_2} + p^{-2s_1 -s_2}}{(1 - p^{-s_1})(1 - p^{-s_1 - s_2})} \right)\prod_{p\in\mathcal{P}_2} \left( \frac{1 - p^{-2s_1 - s_2}}{(1 - p^{-s_1})(1-p^{-s_1-s_2})}\right). \\
  \end{split}
\end{equation}
Putting these into (\ref{eq:cotype3}), we obtain
\begin{equation}
  \label{eq:cotype4}
  \begin{split}
    & \zeta_{\mathbb{Z}[2^{1/3}]}(s_1, s_2, s_3) = \\
    & \quad = (1 + 2^{-s_1} + 2^{-s_1 - s_2})(1 + 3^{-s_1} + 3^{-s_1 - s_2}) \zeta(s_1 + s_2 + s_3) \\
    & \quad \quad  \times \prod_{p\in\mathcal{P}_1} \left( \frac{ 1 + 2p^{-s_1} + 2p^{-s_1 -s_2} + p^{-2s_1 -s_2}}{(1 - p^{-s_1})(1 - p^{-s_1 - s_2})} \right)\prod_{p\in\mathcal{P}_2} \left( \frac{1 - p^{-2s_1 - s_2}}{(1 - p^{-s_1})(1-p^{-s_1-s_2})}\right). \\
  \end{split}
\end{equation}

\subsection{Composition of ideals}
\label{sec:composition}

We start our proof of proposition \ref{proposition:composition} with the following lemma.
\begin{lemma}
  \label{lemma:idealcontain}  
  Let $\mu \pmod m$ and $\nu \pmod n$ satisfy $F(\mu) \equiv 0 \pmod m$ and $F(\nu) \equiv 0 \pmod n$, and let $I$, $J$ be the corresponding ideals via theorem \ref{theorem:correspondence}.
  Then $m \equiv 0 \pmod n$ and $\mu \equiv \nu \pmod n$ if and only if $I \subset J$.
\end{lemma}

\begin{proof}
  We have by the correspondence in theorem \ref{theorem:correspondence} that $I\subset J$ if and only if there is an integral matrix $A$ such that
  \begin{equation}
    \label{eq:latticecontain}
    \begin{pmatrix}
      1 & \cdots & 0 & -\mu^{d-1} \\
      \vdots & \ddots & \vdots & \vdots \\
      0 & \cdots & 1 & -\mu \\
      0 & \cdots & 0 & m
    \end{pmatrix}
    = A
    \begin{pmatrix}
      1 & \cdots & 0 & -\nu^{d-1} \\
      \vdots & \ddots & \vdots & \vdots \\
      0 & \cdots & 1 & -\nu \\
      0 & \cdots & 0 & n
    \end{pmatrix}
    .
  \end{equation}
  Such an $A$ has to have the form
  \begin{equation}
    \label{eq:Aform}
    A = 
    \begin{pmatrix}
      1 & \cdots & 0 & k_{d-1} \\
      \vdots & \ddots & \vdots & \vdots \\
      0 & \cdots & 1 & k_1 \\
      0 & \cdots & 0 & \frac{m}{n}
    \end{pmatrix}
    ,
  \end{equation}
  and so $n\mid m$ is necessary.
  Moreover, integers $k_j$ such that (\ref{eq:latticecontain}) exist if and only if $\mu \equiv \nu \pmod n$, thus proving the lemma.
\end{proof}

This is almost enough already to prove proposition \ref{proposition:composition}, to finish the proof we observe two facts.
First, if $\mu \pmod m$, $\nu \pmod n$ satisfy $F(\mu) \equiv 0 \pmod m$, $F(\nu) \equiv 0 \pmod n$ and $\gcd(m,n) = 1$, then the Chinese remainder theorem gives a unique $\tilde{\mu} \pmod {mn}$ such that $\tilde{\mu} \equiv \mu \pmod m$ and $\tilde{\mu} \equiv \nu \pmod n$.
If $I$, $J$ are the ideals corresponding to $\mu \pmod m$, $\nu \pmod n$ and $\tilde{I}$ is the ideal corresponding to $\tilde{\mu} \pmod {mn}$, then by the lemma $\tilde{I} \subset I \cap J$.
Inspecting the norms of the ideals shows that in fact $\tilde{I} = I \cap J = IJ$ as claimed.

The second observation concerns degree one prime ideals $P$ corresponding to a root $\mu_1 \pmod p$, $p$ not dividing the discriminant of $F$.
If $\mu_k \pmod {p^k}$ is the root given by Hensel's lemma, i.e $\mu_k \pmod {p^k}$ is the unique residue class satisfying $F(\mu_k) \equiv 0 \pmod {p^k}$ and $\mu_k \equiv \mu_1 \pmod p$, then we claim that the ideal $P^k$ corresponds via theorem \ref{theorem:correspondence} to $\mu_k \pmod {p^k}$.
Indeed, since $\mathbb{Z}[\alpha] / P^k$ is additively cyclic, there is a corresponding root modulo $N(P^k) = p^k$, and this root must be $\mu_k \pmod {p^k}$ by lemma \ref{lemma:idealcontain}.
This finishes the proof of proposition \ref{proposition:composition}.

We now move on to proving proposition \ref{proposition:cubiccomposition}.
We remark that proving an extension of lemma \ref{lemma:idealcontain} directly is much more difficult in this setting, at least without first reducing to the case when $m_1$, $m_2$ are powers of the same prime number.
Since it is hard to avoid this reduction we proceed differently by instead working directly with the prime factorization of the ideals themselves.
The following lemma allows us to reduce to the case when the ideals have powers of the same prime as their norm.

\begin{lemma}
  \label{lemma:coprimecubiccomposition}
  Let $I$ and $J$ be ideals not divisible by any rational integers with $\gcd(N(I), N(J)) =1$, and let $\mu_1 \pmod {m_1}$, $\mu_2 \pmod {m_2}$ and $\nu_1 \pmod {n_1}$, $\nu_2 \pmod {n_2}$ be the corresponding roots via theorem \ref{theorem:cubiccorrespondence}, so in particular $N(I) = m_1^2m_2$ and $N(J) = n_1^2 n_2$ are coprime.
  Then $IJ$ is not divisible by rational integers and the corresponding roots are given modulo $m_1n_1$ and $m_2n_2$ by the Chinese remainder theorem applied to $\mu_1 \pmod {m_1}$, $\nu_1 \pmod{n_1}$ and $\mu_2 \pmod {m_2}$, $\nu_2 \pmod {n_2}$.
\end{lemma}

\begin{proof}
  Let $\tilde{\mu}_1 \pmod {m_1n_1}$ and $\tilde{\mu_2} \pmod {m_2n_2}$ be the roots obtained from $\mu_1 \pmod {m_1}$, $\nu_1 \pmod {n_1}$ and $\mu_2 \pmod{m_2}$, $\nu_2 \pmod {n_2}$ by the Chinese remainder theorem.
  Then since $\gcd(m_1n_1, m_2n_2) = \gcd(m_1,m_2)\gcd(n_1,n_2)$ is a factorization into coprime integers, we have $\gcd(m_1n_1, m_2n_2, \tilde{\mu}_1 - \tilde{\mu}_2) = 1$.
  By theorem \ref{theorem:cubiccorrespondence} there exists an ideal, say $\tilde{I}$ that has basis given by
  \begin{equation}
    \label{eq:tildeIbasis}
    \begin{pmatrix}
      1 & \tilde{\mu}_1 + a_1 & \tilde{\lambda} \\
      0 & m_1n_1 & -\tilde{\mu}_2 m_1n_1 \\
      0 & 0 & m_1n_1m_2n_2
    \end{pmatrix}
    \begin{pmatrix}
      \alpha^2 \\
      \alpha \\
      1
    \end{pmatrix}
    .
  \end{equation}
  We claim that $\tilde{I} = IJ$, and we first show that $\tilde{I} \subset I$, which is equivalent to showing that there exists an integral matrix $A$ such that
  \begin{equation}
    \label{eq:cubiclatticecontain}
    \begin{pmatrix}
      1 & \tilde{\mu}_1 + a_1 & \tilde{\lambda} \\
      0 & m_1n_1 & -\tilde{\mu}_2 m_1n_1 \\
      0 & 0 & m_1n_1m_2n_2
    \end{pmatrix}
    = A
    \begin{pmatrix}
      1 & \mu_1 + a_1 & \lambda \\
      0 & m_1  & -\mu_2m_1 \\
      0 & 0 & m_1m_2
    \end{pmatrix}
    .
  \end{equation}

  By changing $A$ as necessary, we can change $\mu_1$ and $\mu_2$ by multiples of $m_1$ and $m_2$ so that in fact $\mu_1 = \tilde{\mu}_1$ and $\mu_2 = \tilde{m}_2$, where we have fixed representatives of $\tilde{\mu}_1 $ and $\tilde{\mu}_2$.
  This of course changes the $\lambda$ given by theorem \ref{theorem:cubiccorrespondence}, but we still denote this new residue class modulo $m_1m_2$ by $\lambda$.
  Having made these replacements, $A$ can be seen to have the form
  \begin{equation}
    \label{eq:Aform3}
    A =
    \begin{pmatrix}
      1 & 0 & * \\
      0 & n_1 & 0 \\
      0 & 0 & n_1n_2
    \end{pmatrix}
    ,
  \end{equation}
  and we observe that the $*$ entry can be chosen so that (\ref{eq:cubiclatticecontain}) holds if $\tilde{\lambda} \equiv \lambda \pmod {m_1m_2}$.

  We recall from (\ref{eq:lambdadef}) that
  \begin{flalign}
    \label{eq:lambdarecall}
    \lambda & \equiv (\tilde{\mu}_1^2 + a_1\tilde{\mu}_1 + a_2) \frac{\overline{m_2}m_2}{\gcd(m_1,m_2)} - (\tilde{\mu}_2^2 + \tilde{\mu}_1\tilde{\mu}_2 + a_1\tilde{\mu}_2 ) \frac{\overline{m}_1m_1}{\gcd(m_1,m_2)} \\ \nonumber
    & \quad + \kappa \frac{m_1m_2}{\gcd(m_1,m_2)} \pmod {m_1m_2}
  \end{flalign}
  where $\overline{m_1}, \overline{m_2}$ are defined by
  \begin{equation}
    \label{eq:m1m2barrecall}
    \frac{\overline{m_1}m_1}{\gcd(m_1,m_2)} + \frac{\overline{m_2}m_2}{\gcd(m_1,m_2)} = 1
  \end{equation}
  and $\kappa$ is defined by
  \begin{equation}
    \label{eq:kapparecall}
    \frac{F(\tilde{\mu}_1)}{m_1} \overline{m_2} + \frac{F(\tilde{\mu}_2)}{m_2} \overline{m_1} + (\tilde{\mu}_1 - \tilde{\mu}_2) \kappa \equiv 0 \pmod {\gcd(m_1,m_2)}.
  \end{equation}
  We also recall that due to the definition of $\kappa$, $\lambda$ is independent of the choice of $\overline{m_1}, \overline{m_2}$.
  In fact we choose
  \begin{equation}
    \label{eq:m1m2barchoice}
    \overline{m_1} = \frac{\overline{m_1n_1}n_1}{\gcd(n_1,n_2)}, \quad \overline{m_2} = \frac{\overline{m_2n_2}n_2}{\gcd(n_1,n_2)},
  \end{equation}
  where $\overline{m_1n_1}, \overline{m_2n_2}$ are so that
  \begin{equation}
    \label{eq:m1n1m2n2bardef}
    \frac{\overline{m_1n_1}m_1n_1}{\gcd(m_1n_1,m_2n_2)} + \frac{\overline{m_2n_2}m_2n_2}{\gcd(m_1n_1,m_2n_2)} = 1.
  \end{equation}
  With this choice, we have $\kappa \equiv \tilde{\kappa} \pmod {\gcd(m_1,n_1)}$, where $\tilde{\kappa}$ is used in the definition of $\tilde{\lambda}$, and we have $\tilde{\lambda} \equiv \lambda \pmod {m_1m_2}$ as required.

  By the same arguments, we have also that $\tilde{I} \subset J$.
  Since the norm of $\tilde{I}$ is $m_1^2n_1^2m_2n_2 = N(I)N(J)$, and since $I$ and $J$ are coprime, we have $\tilde{I} = IJ$, proving the lemma.
\end{proof}

We now consider the setting when a rational prime $p$ factors in $\mathbb{Z}[\alpha]$ as the product of three degree one prime ideals $P_1, P_2, P_3$.
As discussed in the introduction, these degree one prime ideals correspond to three distinct roots $\mu_1, \mu_2, \mu_3 \pmod p$. 

\begin{lemma}
  \label{lemma:primecubiccomposition}
  The ideal $P_2^kP_3^l$ corresponds to $\mu_1$ lifted to a root modulo $p^l$ and $\mu_2$ lifted to a root modulo $p^{k-l}$.
\end{lemma}

\begin{proof}
  Abusing notation slightly, we use $\mu_1, \mu_2, \mu_3$ to denote the lifted roots modulo $p^{2k}$ (or modulo arbitrary powers of primes if one works $p$-adicly).
  Now since $\mu_1, \mu_2$ are distinct, theorem \ref{theorem:cubiccorrespondence} gives an ideal, say $I$, with the basis
  \begin{equation}
    \label{eq:deg1deg1basis}
    \begin{pmatrix}
      1 & \mu_1 + a_1 & \lambda \\
      0 & p^l & -\mu_2 p^l \\
      0 & 0 & p^k
    \end{pmatrix}
    \begin{pmatrix}
      \alpha^2 \\
      \alpha \\
      1
    \end{pmatrix}
    .
  \end{equation}
  The claim is that $I = P_2^kP_3^l$, and to show this we show separately that $I \subset P_2^k$ and $I\subset P_3^l$.
  The claim follows from this since $P_2^k$ and $P_3^l$ are coprime and $N(I) = p^{k+1} = N(P_2^kP_3^l)$.

  We have $I \subset P_2^k$ if and only if there exists an integral matrix $A$ such that
  \begin{equation}
    \label{eq:Adef}
    \begin{pmatrix}
      1 & \mu_1 + a_1 & \lambda \\
      0 & p^l & -\mu_2 p^l \\
      0 & 0 & p^k
    \end{pmatrix}
    = A
    \begin{pmatrix}
      1 & a_1 & -\mu_2^2 - a_1\mu_2 \\
      0 & 1 & - \mu_2 \\
      0 & 0 & p^k
    \end{pmatrix}
    .
  \end{equation}
  Such an $A$ must have the form
  \begin{equation}
    \label{eq:Aform1}
    A =
    \begin{pmatrix}
      1 & \mu_1 & * \\
      0 & p^l & 0 \\
      0 & 0 & 1
    \end{pmatrix}
    ,
  \end{equation}
  and we observe that the $*$ entry can be chosen so that (\ref{eq:Adef}) holds if and only if
  \begin{equation}
    \label{eq:lambdacongruence}
    \lambda \equiv - \mu_1\mu_2 - \mu_2^2 - a_1\mu_2 \pmod {p^k}.
  \end{equation}

  To verify (\ref{eq:lambdacongruence}) we use (\ref{eq:lambdadef}) and consider cases $l \leq k - l$ and $ l > k-l$ separately.
  In the first case we have $\gcd(p^l, p^{k-l}) = p^l$, and so
  \begin{equation}
    \label{eq:lambdarecall1}
    \lambda \equiv -(\mu_2^2 + \mu_1\mu_2 + a_1\mu_2) + \kappa p^{k-l} \pmod {p^k},
  \end{equation}
  where $\kappa $ satisfies
  \begin{equation}
    \label{eq:kapparecall1}
    (\mu_1 - \mu_2) \kappa + \frac{F(\mu_2)}{p^{k-l}} \equiv 0 \pmod {p^l}.
  \end{equation}
  By the way $\mu_2$ was chosen, we have $F(\mu_2) \equiv 0 \pmod {p^k}$, and so $\kappa \equiv 0 \pmod {p^l}$ satisfies (\ref{eq:kapparecall1}) and so (\ref{eq:lambdacongruence}) holds.
  When $l > k-l$, we have $\gcd(p^l, p^{k-1}) = p^{k-l}$ and
  \begin{equation}
    \label{eq:lambdarecall2}
    \lambda \equiv \mu_1^2 + a_1\mu_1 + a_2 + \kappa p^{l} \pmod {p^k}.
  \end{equation}
  This time by the way $\mu_1$ was chosen, we find that $\kappa \equiv 0 \pmod {p^{k-l}}$, and so (\ref{eq:lambdacongruence}) reduces to
  \begin{equation}
    \label{eq:lambdacongruence1}
    \mu_1^2 + \mu_1\mu_2 + \mu_2^2 + a_1(\mu_1 + \mu_2) + a_2 = \frac{F(\mu_1) - F(\mu_2)}{\mu_1 - \mu_2} \equiv 0 \pmod {p^k}.
  \end{equation}
  Since $\mu_1 - \mu_2$ is not zero, (\ref{eq:lambdacongruence1}) indeed holds.

  It now remains to verify that $P_3^l$ contains $I$.
  As before, we have $I \subset P_3^l$ if and only if there is an integral matrix $A$ such that
  \begin{equation}
    \label{eq:Adef1}
    \begin{pmatrix}
      1 & \mu_1 + a_1 & \lambda \\
      0 & p^l & -\mu_2 p^l \\
      0 & 0 & p^k
    \end{pmatrix}
    = A
    \begin{pmatrix}
      1 & a_1 & -\mu_3^2 - a_1\mu_3 \\
      0 & 1 & -\mu_3 \\
      0 & 0 & p^{l}
    \end{pmatrix}
    .
  \end{equation}
  Such an $A$ must have the form
  \begin{equation}
    \label{eq:Aform2}
    A = 
    \begin{pmatrix}
      1 & \mu_1 & * \\
      0 & p^l & (\mu_3 - \mu_2)p^l \\
      0 & 0 & p^{k-l}
    \end{pmatrix}
    ,
  \end{equation}
  and the $*$ entry can be chosen so that (\ref{eq:Adef1}) holds if and only if
  \begin{equation}
    \label{eq:lambdacongruence2}
    \lambda \equiv -\mu_1\mu_3 - \mu_3^2 - a_1\mu_3 \pmod {p^l}.
  \end{equation}
  Again we split into cases $l\leq k-l$ and $l > k-l$.
  In the first case, we have from (\ref{eq:lambdarecall1}) it is enough to verify that
  \begin{equation}
    \label{eq:lambdacongruence3}
    (\mu_1\mu_2 + \mu_2^2 + a_1\mu_2) - (\mu_1\mu_3 + \mu_3^2 + a_1 \mu_3) = (\mu_2 - \mu_3)(\mu_1 + \mu_2 + \mu_3 + a_1) \equiv 0 \pmod {p^l},
  \end{equation}
  which is indeed the case.
  When $l > k-l$, we have from (\ref{eq:lambdarecall2}) that it is enough to verify
  \begin{equation}
    \label{eq:lambdacongruence4}
    \mu_1^2 + \mu_1\mu_3 + \mu_3^2 + a_1(\mu_1 + \mu_3) + a_2 = \frac{F(\mu_1) - F(\mu_3)}{\mu_1 - \mu_3} \equiv 0 \pmod {p^l},
  \end{equation}
  which is again the case.
\end{proof}

\newpage

\bibliographystyle{plain}
\bibliography{references.bib}

\begin{thebibliography}{10}

\bibitem{BumpFriedbergGoldfeld1988}
Daniel Bump, Solomon Friedberg, and Dorian Goldfeld.
\newblock Poincaré series and kloosterman sums for sl(3, z).
\newblock {\em Acta Arithmetica}, 50(1):31--89, 1988.

\bibitem{Buttcane2012}
Jack Buttcane.
\newblock {\em Sums of {SL}(3,{Z}) {K}loosterman {S}ums}.
\newblock ProQuest LLC, Ann Arbor, MI, 2012.
\newblock Thesis (Ph.D.)--University of California, Los Angeles.

\bibitem{Bykovskii1984}
V.~A. Bykovski\u{\i}.
\newblock Spectral expansions of certain automorphic functions and their
  number-theoretic applications.
\newblock volume 134, pages 15--33. 1984.
\newblock Automorphic functions and number theory, II.

\bibitem{CheungChevallier2016}
Yitwah Cheung and Nicolas Chevallier.
\newblock Hausdorff dimension of singular vectors.
\newblock {\em Duke Math. J.}, 165(12):2273--2329, 09 2016.

\bibitem{ChintaKaplanKoplewitz2017}
Gautam Chinta, Nathan Kaplan, and Shaked Koplewitz.
\newblock The cotype zeta function of $\mathbb{Z}^d$.
\newblock Submitted, 2017.

\bibitem{ConradNote}
Keith Conrad.
\newblock The different ideal.
\newblock \url{https://kconrad.math.uconn.edu/blurbs/gradnumthy/different.pdf}.

\bibitem{Dartyge2015}
C\'{e}cile Dartyge.
\newblock Le probl\`eme de {T}ch\'{e}bychev pour le douzi\`eme polyn\^{o}me
  cyclotomique.
\newblock {\em Proc. Lond. Math. Soc. (3)}, 111(1):1--62, 2015.

\bibitem{delaBreteche2015}
R.~de~la Bret\`eche.
\newblock Plus grand facteur premier de valeurs de polyn\^{o}mes aux entiers.
\newblock {\em Acta Arith.}, 169(3):221--250, 2015.
\newblock With an appendix by de la Bret\`eche and J.-F. Mestre.

\bibitem{DukeFriedlanderIwaniec1995}
W.~Duke, J.~B. Friedlander, and H.~Iwaniec.
\newblock Equidistribution of roots of a quadratic congruence to prime moduli.
\newblock {\em Ann. of Math. (2)}, 141(2):423--441, 1995.

\bibitem{DukeFriedlanderIwaniec2012}
W.~Duke, J.~B. Friedlander, and H.~Iwaniec.
\newblock Weyl sums for quadratic roots.
\newblock {\em Int. Math. Res. Not.}, (11):2493--2549, 2012.

\bibitem{FouvryIwaniec1997}
Etienne Fouvry and Henryk Iwaniec.
\newblock Gaussian primes.
\newblock {\em Acta Arith.}, 79(3):249--287, 1997.

\bibitem{FriedlanderIwaniec1998}
John Friedlander and Henryk Iwaniec.
\newblock Asymptotic sieve for primes.
\newblock {\em Ann. of Math. (2)}, 148(3):1041--1065, 1998.

\bibitem{Heath-Brown2000}
D.~R. Heath-Brown.
\newblock The largest prime factor of $x^3 - 2$.
\newblock {\em Proceedings of the London Mathematical Society},
  82(3):554–596, 2000.

\bibitem{Heath-Brown2001}
D.~R. Heath-Brown.
\newblock Primes represented by $x^3 + 2y^3$.
\newblock {\em Acta Math.}, 186(1):1--84, 2001.

\bibitem{Hejhal1986}
Dennis~A. Hejhal.
\newblock Roots of quadratic congruences and eigenvalues of the non-{E}uclidean
  {L}aplacian.
\newblock In {\em The {S}elberg trace formula and related topics ({B}runswick,
  {M}aine, 1984)}, volume~53 of {\em Contemp. Math.}, pages 277--339. Amer.
  Math. Soc., Providence, RI, 1986.

\bibitem{Hooley1964}
C.~Hooley.
\newblock On the distribution of the roots of polynomial congruences.
\newblock {\em Mathematika}, 11:39--49, 1964.

\bibitem{Hooley1963}
Christopher Hooley.
\newblock On the number of divisors of a quadratic polynomial.
\newblock {\em Acta Math.}, 110:97--114, 1963.

\bibitem{Hooley1978}
Christopher Hooley.
\newblock On the greatest prime factor of a cubic polynomial.
\newblock {\em J. Reine Angew. Math.}, 303/304:21--50, 1978.

\bibitem{Iwaniec1978}
Henryk Iwaniec.
\newblock Almost-primes represented by quadratic polynomials.
\newblock {\em Invent. Math.}, 47(2):171--188, 1978.

\bibitem{KowalskiSoundararajan2020}
Emmanuel Kowalski and Kannan Soundararajan.
\newblock Equidistribution from the {C}hinese {R}emainder {T}heorem, 2020.

\bibitem{LubotzkySegal2003}
Alexander Lubotzky and Dan Segal.
\newblock {\em Subgroup growth}, volume 212 of {\em Progress in Mathematics}.
\newblock Birkh\"{a}user Verlag, Basel, 2003.

\bibitem{Petrogradsky2007}
V.~M. Petrogradsky.
\newblock Multiple zeta functions and asymptotic structure of free abelian
  groups of finite rank.
\newblock {\em J. Pure Appl. Algebra}, 208(3):1137--1158, 2007.

\bibitem{Sarnak1990}
Peter Sarnak.
\newblock {\em Some applications of modular forms}, volume~99 of {\em Cambridge
  Tracts in Mathematics}.
\newblock Cambridge University Press, Cambridge, 1990.

\bibitem{Terras1988}
Audrey Terras.
\newblock {\em Harmonic analysis on symmetric spaces and applications. {II}}.
\newblock Springer-Verlag, Berlin, 1988.

\bibitem{Toth1997}
Arpad Toth.
\newblock {\em Equidistribution of roots of quadratic congruences}.
\newblock ProQuest LLC, Ann Arbor, MI, 1997.
\newblock Thesis (Ph.D.)--Rutgers The State University of New Jersey - New
  Brunswick.

\end{thebibliography}

\end{document}